\documentclass[a4wide,11pt]{article}

\usepackage{graphicx}
\usepackage{color}
\usepackage{version}
\usepackage[colorlinks,citecolor=blue]{hyperref}
\usepackage[top=1.7cm,bottom=2.5cm,left=1.9cm,right=1.9cm]{geometry}
\usepackage{array}

\usepackage{lcbmaths}
\usepackage{mathtools}
\usepackage{gcsreg}

\usepackage{natbib}

\newcommand{\footremember}[2]{%
    \footnote{#2}
    \newcounter{#1}
    \setcounter{#1}{\value{footnote}}%
}
\newcommand{\footrecall}[1]{%
    \footnotemark[\value{#1}]%
}

\title{Sampling distribution for single-regression Granger causality estimators}

\author{%
A. J. Gutknecht\footremember{Juelich}{Institute of Neuroscience and Medicine (INM-6), Forschungszentrum Juelich, Germany.}%
\footremember{MEG_unit}{MEG Unit, Brain Imaging Center, Goethe University, Frankfurt am Main, Germany.}%
\footremember{sackler}{Sackler Centre for Consciousness Science and Dept. of Informatics, University of Sussex, Falmer, Brighton, UK}%
\ and L. Barnett\footrecall{sackler} \footremember{corrauth}{Corresponding author: \texttt{\href{mailto:l.c.barnett@sussex.ac.uk}{l.c.barnett@sussex.ac.uk}}\vspace{0.5em}}%
}

\begin{document}

\date\today

\maketitle

\vspace{-1em}

\begin{abstract}
We show for the first time that, under the null hypothesis of vanishing Granger causality, the single-regression Granger-Geweke estimator converges to a generalised $\chi^2$ distribution, which may be well approximated by a $\Gamma$ distribution. We show that this holds too for Geweke's spectral causality averaged over a given frequency band, and derive explicit expressions for the generalised $\chi^2$ and $\Gamma$-approximation parameters in both cases. We present an asymptotically valid Neyman-Pearson test based on the single-regression estimators, and discuss in detail how it may be usefully employed in realistic scenarios where autoregressive model order is unknown or infinite. We outline how our analysis may be extended to the conditional case, point-frequency spectral Granger causality, state-space Granger causality, and the Granger causality $F$-test statistic. Finally, we discuss approaches to approximating the distribution of the single-regression estimator under the alternative hypothesis.
\end{abstract}

\section{Introduction}
Since its inception in the 1960s, Wiener-Granger causality (GC) has found many applications in a range of disciplines, from econometrics, neuroscience, climatology, ecology, and beyond. In the early 1980s Geweke introduced the standard vector-autoregressive (VAR) formalism, and the Granger-Geweke population log-likelihood-ratio (LR) statistic \citep{Geweke:1982,Geweke:1984}. As well as furnishing a likelihood-ratio test for statistical significance, the statistic has been shown to have an intuitive information-theoretic interpretation as a quantitative measure of ``information transfer'' between stochastic processes \citep{Barnett:tegc:2009,Barnett:teml:2012}. In finite sample, the LR estimator requires separate estimates for the full and reduced VAR models, and as such admits the classical large-sample theory, and asymptotic $\chi^2$ distribution \citep{NeymanPearson:1933,Wilks:1938,Wald:1943}. However, it has become increasingly clear that the ``dual-regression'' LR estimator is problematic: specifically, model order selection involves a bias-variance trade-off which may skew statistical inference, impact efficiency, and produce spurious results, including negative GC values \citep{DingEtal:2006,Chen:2006,StokesPurdon:2017}.

An alternative \emph{single-regression} (SR) estimator which obviates the problem has been developed in various forms over the past decade \citep{Dhamala:2008b,Dhamala:2008a,Barnett:mvgc:2014,Barnett:ssgc:2015,Solo:2016}; but, since the large-sample theory no longer obtains, its sampling distribution has remained unknown until now. In addition to the improved efficiency and reduced bias of the SR estimator \citep{Barnett:mvgc:2014,Barnett:ssgc:2015}, knowledge of its sampling distribution under the null hypothesis of vanishing causality would allow to construct novel and potentially superior hypothesis tests, especially in the frequency domain where little is known about the sampling distribution of Geweke's spectral GC statistic\footnote{This applies even in the unconditional case where the Geweke spectral statistic \citep{Geweke:1982} requires only a single estimate of the full model.} \citep{Geweke:1982,Geweke:1984}. Closing this gap is thus the central object of the present study.

We begin in \secref{sec:prelim} with an overview of the theoretical and technical prerequisites of maximum-likelihood (ML) estimation, large-sample theory, the generalised $\chi^2$ distribution, VAR modelling, and Granger causality estimation. Since SR estimators (both time-domain and frequency-band-limited) are continuous functions of the ML estimator for the full VAR model, the problem of determining their asymptotic distributions can be approached via (a second-order version of) the well-known Delta Method \citep{LehmannRomano:2005}, in conjunction with a state-space spectral factorisation technique \citep{Wilson:1972,HandD:2012} which allows explicit derivation of reduced-model parameters in terms of the full-model parameters \citep{Barnett:ssgc:2015,Solo:2016}. In \secref{sec:asymdist} we show in this way that the asymptotic distributions are generalised $\chi^2$, and derive explicit expressions for the distributional parameters in terms of the parameters of the underlying VAR model.

Unlike under the classical theory, the asymptotic null distributions of the SR estimators have a dependence on the true null parameters, which presents a challenge for statistical inference. These issues are addressed in \secref{sec:statinf}, in which we present an asymptotically valid Neyman-Pearson test that accommodates this dependence. We discuss the properties of this test in terms of type I and type II error probabilities. In \secref{sec:disc} we complete our analysis by addressing the issues of model order selection and potentially infinite model orders. We elaborate on how our analysis might be extended in the future to the conditional case  \citep{Geweke:1984}, point-frequency spectral GC \citep{Geweke:1982,Geweke:1984}, state-space GC \citep{Barnett:ssgc:2015,Solo:2016}, and the GC $F$-statistic \citep[Chap. 11]{Hamilton:1994}. Finally, we discuss approaches to approximating the distribution of the SR estimator under the alternative hypothesis.

\section{Preliminaries} \label{sec:prelim}

\textit{Notation:} Throughout, superscript ``$\trop$'' denotes matrix transpose and superscript ``$*$'' conjugate transpose, while superscript ``$\rrop$'' refers to a reduced model (\secref{sec:gc}); $|\cdots|$ denotes the determinant of a a square matrix and $\trace\cdots$ its trace; $\vnorm\cdots$ denotes a consistent vector/matrix norm; $\prob\cdots$ denotes probability, $\expect\cdots$ expectation and $\var\cdots$ variance; $\pconverge$ denotes convergence in probability and $\dconverge$ convergence in distribution; ``$\log$'' denotes natural logarithm. Unless stated otherwise, all vectors are taken to be \emph{column} vectors.

\subsection{Maximum likelihood estimation and the large-sample theory} \label{sec:mle}

Suppose given a parametrised set of models on a state space $\sspace = \reals^n$, with multivariate parameter $\btheta = [\theta_1,\ldots,\theta_r]^\trop$ in an $r$-dimensional parameter space $\Theta \subseteq \reals^r$. A data sample of size $N$ for the model is then a sequence $\bu = \{\bu_1,\bu_2,\ldots,\bu_N\} \in \sspace^N$; in general, the $\bu_k$ will not be independent, and we denote the joint distribution of data samples of size $N$ drawn from the model with parameter $\btheta$ by $\sdist_N(\btheta)$, with probability density function (PDF) $p(\bu;\btheta)$. For sample data $\bu \in \sspace^N$, the \emph{average log-likelihood function} is defined to be\footnote{We work exclusively with average log-likelihoods, so from here on we drop the ``average log-''.}
\begin{equation}
	\mllhood(\btheta|\bu) = \frac1N \log p(\bu;\btheta)
\end{equation}
\ie, the logarithm of the PDF scaled by sample size. The \emph{maximum likelihood} and \emph{maximum-likelihood parameter estimate} are given respectively by
\begin{align}
	\mllhood(\bu) &= \sup\big\{\mllhood(\btheta|\bu) : \btheta \in \Theta\big\} \\
	\hbtheta(\bu) &\equiv \argmax\big\{\mllhood(\btheta|\bu) : \btheta \in \Theta\big\}
\end{align}
We assume the model is identifiable, so that $\hbtheta(\bu)$ is uniquely defined and $\mllhood(\bu) = \mllhood\big(\hbtheta(\bu)|\bu\big)$. On a point of notation, for any sample statistic $\hat f(\bu)$, given $\btheta \in \Theta$ we write $\hat f(\btheta)$ for the random variable $\hat f(\bU)$ with $\bU \sim \sdist_N(\btheta)$; $\hat f(\btheta)$ is thus a random variable parametrised by $\btheta$. For a parameter $\btheta$ itself, we write just $\hbtheta$ for its ML estimator $\hbtheta(\bU)$.

Let
\begin{equation}
	\mllhood_\alpha(\btheta|\bu) = \frac{\partial}{\partial\theta_\alpha} \mllhood(\btheta|\bu)\,, \qquad \mllhood_{\alpha\beta}(\btheta|\bu) = \frac{\partial^2}{\partial\theta_\alpha\partial\theta_\beta} \mllhood(\btheta|\bu)\,, \qquad \alpha,\beta = 1,\ldots,r
\end{equation}
The \emph{Fisher information matrix} associated with a parameter $\btheta \in \Theta$ is defined as
\begin{equation}
	\finfo_{\alpha\beta}(\btheta) \equiv -\expect{\mllhood_{\alpha\beta}(\btheta|\bU)} = -\int \mllhood_{\alpha\beta}(\btheta|\bu) p(\bu;\btheta) \,d\bu_1\ldots d\bu_N
\end{equation}
\ie, expectation is over $\bU \sim \sdist_N(\btheta)$. Writing $\pcmat(\btheta) = \finfo(\btheta)^{-1}$ for the inverse Fisher information matrix, under certain conditions---which will apply in our case---we have $\forall \btheta \in \Theta$
\begin{align}
	\hbtheta &\pconverge \btheta \label{eq:mlpcons} \\
	\sqrt N\,\big(\hbtheta - \btheta\big) &\dconverge \snormal\big(\bzero,\pcmat(\btheta)\big) \label{eq:mlpdist}
\end{align}
as sample size $N \to \infty$; that is, the maximum-likelihood parameter estimate is consistent and (appropriately scaled) asymptotically normally distributed with mean equal to the true parameter, and covariance given by the inverse Fisher information matrix.

Suppose now that we have a composite nested null hypothesis $H_0$ defined by $\btheta \in \Theta_0$, where $\Theta_0 \subset \Theta$ is the $s$-dimensional \emph{null subspace}, $s < r$. We write the maximum likelihood and maximum-likelihood parameter under $H_0$ as
\begin{align}
	\mllhood_0(\bu) &= \sup\big\{\mllhood(\btheta|\bu) : \btheta \in \Theta_0\big\} \\
	\hbtheta_0(\bu) &\equiv \argmax\big\{\mllhood(\btheta|\bu) : \btheta \in \Theta_0\big\}
\end{align}
respectively. The \emph{(log-)likelihood ratio} statistic given the data sample $\bu$ is then
\begin{equation}
	\llrat(\bu) \equiv 2\bracs{\mllhood(\bu) - \mllhood_0(\bu)}
\end{equation}
Note that $\mllhood(\bu) \ge \mllhood_0(\bu)$ always, so $\llrat(\bu) \ge 0$. Given $\btheta$, we write, as described earlier, $\llrat(\btheta)$ for the $\btheta$-parametrised random variable $\llrat(\bU)$ with $\bU \sim \sdist_N(\btheta)$. Wilks' Theorem \citep{Wilks:1938} then states that under $H_0$, the distribution of the sample log-likelihood ratio $\llrat(\btheta)$ is asymptotically $\chi^2$-distributed:
\begin{equation}
	\forall \btheta \in \Theta_0\,, \quad N \llrat(\btheta) \dconverge \chi^2(d) \label{eq:wilks}
\end{equation}
as sample size $N \to \infty$, where degrees of freedom $d = r-s$ is the difference in dimension of the full and null parameter spaces. Convergence is of order $N^{-\frac12}$. Crucially, \eqref{eq:wilks} holds for \emph{any} $\btheta \in \Theta_0$; \ie, given that the null hypothesis holds, it doesn't matter \emph{where} in the null space the true parameter lies.

\subsection{The generalised \texorpdfstring{$\chi^2$}{chi-squared} family of distributions} \label{sec:genchi2}

We introduce a family of distributions that will play a critical role in what follows. Let $\bZ \sim \snormal(\bzero,B)$ be a zero-mean $n$-dimensional multivariate-normal random vector with covariance matrix $B$, and $A$ an $n \times n$ symmetric matrix. Then \citep{Jones:1983} we write $\chi^2(A,B)$ for the distribution of the random quadratic form $Q = \bZ^\trop\!A \bZ$. If $A = B = I$, then $\chi^2(A,B)$ reduces to the usual $\chi^2(n)$. If $A$ is $m \times m$ and $C$ is $m \times n$, then $\chi^2(A,CBC^\trop) = \chi^2(C^\trop AC,B)$.

It is not hard to show \citep{Mohsenipour:2012} that if $B$ is positive-definite and $A$ symmetric (which will be the case for the generalised $\chi^2$ distributions we encounter), then $\chi^2(A,B) = \chi^2(\Lambda,I)$, where $\Lambda = \diag{\lambda_1,\ldots,\lambda_n}$ with $\lambda_1,\ldots,\lambda_n$ the eigenvalues of $BA$, or, equivalently, of $RAR^\trop$ where $R$ is the right-Cholesky factor of $B$ (so that $
R^\trop R = B$).  In that case, we have
\begin{equation}
	\lambda_1 U_1^2 + \ldots + \lambda_n U^2_n \sim \chi^2(A,B)\,, \qquad\text{where } U_i \text{ iid} \sim \snormal(0,1) \label{eq:gchi2}
\end{equation}
so that $\chi^2(A,B)$ is a weighted sum of independent $\chi^2$-distributed variables, and in particular if the $\lambda_i$ are all equal then we have a scaled $\chi^2(n)$ distribution. From \eqref{eq:gchi2}, moments of a generalised $\chi^2$ variable may be conveniently expressed in terms of the eigenvalues; thus we may calculate that for $Q \sim \chi^2(A,B)$
\begin{subequations}
\begin{align}
	\expect{Q} &= \mu\,\; = \phantom2 \sum_{i=1}^n \lambda_i    \label{eq:gc2mean} \\
	\var{Q}    &= \sigma^2 = 2\sum_{i=1}^n \lambda_i^2 \label{eq:gc2var}
\end{align} \label{eqs:gc2meanvar}%
\end{subequations}
Empirically, it is found that generalised $\chi^2$ variables (at least for $A$ symmetric and $B$ positive-definite) are very well approximated by $\Gamma$ distributions: specifically, we have $Q \approx \Gamma(\alpha,\beta)$ with
\begin{subequations}
\begin{alignat}{3}
	\alpha &= \frac{\mu^2}{\sigma^2} && \qqquad\text{(shape parameter)} \label{eq:gamparmsa} \\
	\beta &= \frac{\sigma^2}{\mu}   && \qqquad\text{(scale parameter)} \label{eq:gamparmsb}
\end{alignat} \label{eqs:gamparms}%
\end{subequations}

\subsection{VAR modelling} \label{sec:var}
Given a wide-sense stationary, purely nondeterministic $n$-dimensional vector stochastic process $\bU_t = [U_{1t},\allowbreak\ldots,U_{nt}]^\trop$, $-\infty < t < \infty$ \citep{Doob:1953}, under certain conditions (see below) it will have a stable, invertible---and in general infinite-order---VAR representation
\begin{equation}
	\bU_t = \sum_{k=1}^\infty A_k\bU_{t-k} + \beps_t \label{eq:var}
\end{equation}
where $\beps_t$ is a white noise process, the sequence of $n \times n$ autoregression coefficient matrices $A_k$ is square-summable, and the $n \times n$ residuals covariance matrix $\Sigma = \expect{\beps_t \beps_t^\trop}$ is positive-definite. Sufficient conditions for existence of such an invertible VAR representation---and which, importantly, also guarantee a similar representation for any \emph{sub}process---are given in \citet{Geweke:1982} \citep[see also][]{Masani:1966,Rozanov:1967}; we assume those conditions for all vector stochastic processes from now on.

If $A_k = 0$ for $k > p$ then \eqref{eq:var} defines a finite-order VAR($p$) model. We write $A = [A_1 \ldots A_p]$ (an $n \times pn$ matrix), and the model parameters are given by $\btheta = (A,\Sigma)$. The dimensionality of the parameter space is thus $pn^2 + \frac12 n(n+1)$. The autocovariance sequence for the process $\bU_t$ is defined by
\begin{equation}
	\Gamma_k = \expect{\bU_t \bU_{t-k}^\trop}\,, \qquad -\infty < k < \infty\,, \label{eq:acseq}
\end{equation}
and we have $\Gamma_{-k} = \Gamma_k^\trop$. By a standard trick, the process $\bU^{(p)}_t = \big[\bU_t^\trop \; \bU_{t-1}^\trop \ldots \bU_{t-(p-1)}^\trop\big]^\trop$ satisfies a $pn$-dimensional VAR($1$) model
\begin{equation}
	\bU^{(p)}_t = \bA \bU^{(p)}_{t-1} + \beps^{(p)}_t \label{eq:vartrick}
\end{equation}
where the $pn \times pn$ ``companion matrix'' is given by
\begin{equation}
	\bA =
	\begin{bmatrix}
		A_1 & A_2 & \ldots & A_{p-1} & A_p \\
		I   & 0   & \ldots & 0       & 0   \\
		0   & I   & \ldots & 0       & 0   \\
		\vdots & \vdots & \ddots & \vdots & \vdots   \\
		0   & 0   & \ldots & I       & 0
	\end{bmatrix} \qquad\quad \label{eq:compA}
\end{equation}
and  $\beps^{(p)}_t = \big[\beps_t^\trop \; 0 \ldots 0\big]^\trop$ with $pn \times pn$ covariance matrix
\begin{equation}
	\bSigma =
	\begin{bmatrix}
		\Sigma & 0 & \ldots & 0 \\
		0      & 0 & \ldots & 0 \\
		\vdots & \vdots & \ddots & \vdots \\
		0      & 0 & \ldots & 0
	\end{bmatrix} \qquad\quad \label{eq:compSig}
\end{equation}
The \emph{spectral radius} of the model is given by the largest absolute eigenvalue of $\bA$:
\begin{equation}
	\rho(A) = \max\{|z| : |Iz-\bA| = 0\} \label{eq:specrad}
\end{equation}
The model is stable iff $\rho(A) < 1$.

Taking the covariance of both sides of \eqref{eq:vartrick} yields
\begin{equation}
	\bGamma - \bA \bGamma \bA^\trop = \bSigma \label{eq:gamlyap}
\end{equation}
where $\bGamma$ is the $pn \times pn$ covariance matrix
\begin{equation}
	\bGamma = \expect{\bU^{(p)}_t \bU^{(p)\trop}_t} \label{eq:bGamma}
\end{equation}
The $k\ell$-block of $\bGamma$ is given by $\bGamma_{k\ell} = \Gamma_{\ell-k}$ for $k,\ell = 1,\ldots,p$. \eqref{eq:gamlyap} is a discrete-time Lyapunov (DLYAP) equation---which may be readily solved numerically---and is equivalent to (the first $p$ of) the Yule-Walker equations
\begin{equation}
	\Gamma_k = \sum_{\ell=1}^p A_\ell \Gamma_{k-\ell} + \delta_{k0} \Sigma \qquad -\infty < k < \infty \label{eq:yw}
\end{equation}
If the parameters $(A,\Sigma)$ are known, the autocovariance sequence may be calculated from \eqref{eq:gamlyap}, or recursively from \eqref{eq:yw}; conversely, $(A,\Sigma)$ may be calculated from $\Gamma_0,\ldots,\Gamma_p$, \eg, by Whittle's algorithm \citep{Whittle:1963}. In sample, given a data sequence $\bu = \{\bu_1,\ldots,\bu_N\}$, ML parameters $\big(\hA(\bu),\hSigma(\bu)\big)$ may be calculated via a standard ordinary least squares (OLS).

In the spectral domain \citep{HandD:2012}, let $\omega \in [0,2\pi]$ denote angular frequency in radians. The \emph{transfer function} for the VAR model \eqref{eq:var} is defined as
\begin{equation}
	\Psi(\omega) = \Phi(\omega)^{-1} \label{eq:trfun}
\end{equation}
where
\begin{equation}
	\Phi(\omega) = I-\sum_{k = 1}^\infty A_k e^{-i\omega k} \label{eq:itrfun}
\end{equation}
is the Fourier transform of the VAR coefficients sequence. The \emph{cross-power spectral density} (CPSD) matrix $S(\omega)$ is given by the Fourier transform of the autocovariance sequence
\begin{equation}
	S(\omega) = \sum_{k = -\infty}^\infty \Gamma_k e^{-i\omega k} \label{eq:cpsd}
\end{equation}
and conversely,  the autocovariance sequence is the inverse transform of the CPSD:
\begin{equation}
	\Gamma_k = \frac1{2\pi} \int_0^{2\pi} S(\omega) e^{i\omega k} d\omega \label{eq:acov}
\end{equation}
$S(\omega)$ is Hermitian $\forall\omega$, and satisfies the factorisation \citep{WienerMasani:1957}
\begin{equation}
	S(\omega) = \Psi(\omega) \Sigma \Psi(\omega)^* \label{eq:specfac}
\end{equation}
The CPSD $S(\omega)$, via \eqref{eq:specfac},  uniquely determines the VAR parameters, which may be factored out computationally, \eg, by Wilson's algorithm \citep{Wilson:1972}.

\subsection{Granger-Geweke causality} \label{sec:gc}

\citet{Geweke:1982} defines the population (unconditional\footnote{In this study we only address the unconditional case; but see \secref{sec:extend}.})  Granger causality statistic in the following context: suppose that the process \eqref{eq:var} is partitioned into subprocesses $\bU_t = [\bX_t^\trop\;\bY_t^\trop]^\trop$ of dimension $n_x,n_y$ respectively. The assumed regularity conditions on $\bU_t$ \citep[Sec.~2]{Geweke:1982} ensure that the subprocess $\bX_t$ will itself admit a stable, invertible VAR representation
\begin{equation}
	\bX_t = \sum_{k=1}^\infty A^\rrop_k\bX_{t-k} + \beps^\rrop_t \label{eq:varr}
\end{equation}
with square-summable coefficients $A^\rrop_k$ and positive-definite residuals covariance matrix $\Sigma^\rrop = \expect{\beps^\rrop_t \beps^{\rrop\trop}_t}$. To define Granger causality, the \emph{reduced} regression \eqref{eq:varr} is contrasted with the \emph{full} regression; that is, the $x$-component
\begin{equation}
	\bX_t = \sum_{k=1}^\infty A_{k,xx} \bX_{t-k} + \sum_{k=1}^\infty A_{k,xy} \bY_{t-k} + \beps_{xt} \label{eq:varf}
\end{equation}
of \eqref{eq:var}. We stress here that:
\begin{enumerate}
	\item The reduced model parameters $(A^\rrop,\Sigma^\rrop)$ are fully determined by the full model parameters $(A,\Sigma)$.
	\item Even if the full regression \eqref{eq:varf} has finite order, the reduced regression \eqref{eq:varr} will in general \emph{not} have finite order.
\end{enumerate}
The population Granger causality from $\bY \to \bX$ for the VAR \eqref{eq:var} with parameters $\btheta$ is then defined as
\begin{equation}
	F_{\bY\to\bX}(\btheta) = \log\frac{\big|\Sigma^\rrop\big|}{|\Sigma_{xx}|} \label{eq:popgc}
\end{equation}
The justification for the description of \eqref{eq:popgc} as a ``causal'' statistic, is as follows: $\beps_{xt}$ in \eqref{eq:varf} represents the residual error associated with the optimal least-squares prediction $\cexpect{\bX_t}{\bU_{t-1},\bU_{t-2},\ldots} = \sum_{k=1}^\infty A_{k,xx} \bX_{t-k} + \sum_{k=1}^\infty A_{k,xy} \bY_{t-k} $ of $\bX_t$ by its own past $\bX_{t-1},\bX_{t-2},\ldots$ and the past $\bY_{t-1},\bY_{t-2},\ldots$ of $\bY_t$. By contrast, $\beps^\rrop_t$ in \eqref{eq:varr} represents the residual error associated with the optimal prediction $\cexpect{\bX_t}{\bX_{t-1},\bX_{t-2},\ldots} = \sum_{k=1}^\infty A^\rrop_k\bX_{t-k}$ of $\bX_t$ by its own past alone. Magnitudes of residual prediction errors are then quantified by the generalised variances $|\Sigma_{xx}|$ and $\big|\Sigma^\rrop\big|$ respectively \citep{Wilks:1932,Barrett:2010}, leading to the interpretation of \eqref{eq:popgc} as the extent to which inclusion of the past $\bY_{t-1},\bY_{t-2},\ldots$ of $\bY_t$ in the predictor set reduces the prediction error for $\bX_t$.

In the frequency domain, \citet{Geweke:1982} defines the (population, unconditional) spectral Granger causality at angular frequency $\omega$ by
\begin{equation}
	f_{\bY\to\bX}(\omega;\btheta) = \log\frac{|S_{xx}(\omega)|}{\big| S_{xx}(\omega) - \Psi_{xy}(\omega) \Sigma_{yy|x} \Psi_{xy}(\omega)^*\big|} \label{eq:sgc}
\end{equation}
\citet{Barnett:gcfilt:2011} introduce \emph{band-limited} (frequency-averaged) spectral Granger causality
\begin{equation}
	f_{\bY\to\bX}(\frange;\btheta) = \frac1{|\frange|} \int_\frange f_{\bY\to\bX}(\omega;\btheta)\,d\omega \label{eq:sgcbl}
\end{equation}
where the frequency range $\frange$ is a measurable subset of $[0,2\pi]$ (in practice usually an interval). $f_{\bY\to\bX}(\omega;\btheta)$ averaged across all frequencies yields the corresponding time-domain GC \citep{Geweke:1982}; that is,
\begin{equation}
	f_{\bY\to\bX}([0,2\pi];\btheta) = \frac1{2\pi} \int_0^{2\pi} f_{\bY\to\bX}(\omega;\btheta)\,d\omega = F_{\bY\to\bX}(\btheta) \label{sgcint}
\end{equation}

\subsubsection{Likelihood-ratio estimation} \label{sec:gcmle}

Suppose given a finite-order VAR model.
\begin{equation}
	\bU_t = \sum_{k=1}^p A_k\bU_{t-k} + \beps_t \label{eq:varp}
\end{equation}
for the process $\bU_t = [\bX_t^\trop\;\bY_t^\trop]^\trop$. For now, we assume that the model order $p$ is known (in \secref{sec:infestmod} we discuss infinite-order VAR models and model order selection). How then, given a data sequence $\bu = \{\bu_1,\ldots,\bu_N\}$ sampled from \eqref{eq:varp} with unknown parameters $\btheta = (A,\Sigma)$, might $F_{\bY\to\bX}(\btheta)$ be \emph{estimated}? On the face of it, $F_{\bY\to\bX}(\btheta)$ is---uncoincidentally, as remarked in \citet{Geweke:1982}---a population log-likelihood-ratio statistic, since the maximum (average) log-likelihood for a finite-order VAR model of the form \eqref{eq:varp} is, up to a constant additive term, just $-\frac12 \log|\hSigma(\bu)|$, where $\hSigma(\bu)$ is the ML (OLS) estimate for the actual residuals covariance matrix $\Sigma$. But as regards estimation of $F_{\bY\to\bX}(\btheta)$, a problem arises: while the full model order $p$ may be finite, the reduced model \eqref{eq:varr} will in general be of \emph{infinite} order. We might be tempted---as suggested in \citet{Geweke:1982,Geweke:1984}, and, until recently, standard practice---to simply truncate the reduced model at  order $p$. Then \eqref{eq:varr} becomes a nested sub-model of \eqref{eq:varf}, and plugging in maximum-likelihood estimates $\hSigma(\bu)$, $\hSigma^\rrop(\bu)$ for the residuals covariance matrices $\Sigma$, $\Sigma^\rrop$  in \eqref{eq:popgc} yields a true log-likelihood-ratio sample statistic
\begin{equation}
	\hF^\lre_{\bY\to\bX}(\bu) = \log\frac{\big|\hSigma^\rrop(\bu)\big|}{|\hSigma_{xx}(\bu)|}
	\label{eq:sampgclr}
\end{equation}
for the composite null hypothesis
\begin{equation}
	H_0 : A_{1,xy} = \ldots = A_{p,xy} = 0 \label{eq:H0p}
\end{equation}
But here a problem arises: the resulting reduced model will be misspecified, and failure to take into account sufficient lags of $\bX_t$ in the reduced regression biases the resulting Granger causality estimator. Noting that a VAR($p$) model is also VAR($q$) for $q > p$ (coefficients $A_k, p < k \le q$ can simply be set to zero), we could attempt to remedy the situation by selecting a parsimonious model order $q > p$ for the reduced model by a standard (data sample length-dependent) model order selection criterion \citep{McQuarrie:1998}, and extend the full model to order $q$. However, in doing so the full model becomes over-specified and the variance of the resulting estimator is inflated. Furthermore, since the estimated model order will increase with sample length $N$, it is not clear whether the estimator will be consistent in any meaningful sense. We discuss this further in \secref{sec:infestmod}. This conundrum was explicitly identified by \citet{StokesPurdon:2017}\footnote{\citet{StokesPurdon:2017}, having identified the LR estimator as problematic, concede that at the time they were unaware that there were already estimators which obviate the problem \citep{StokesPurdon:2018}. Other claims made in \citet{StokesPurdon:2017} have also come under critical scrutiny \citep{FaesEtal:2017,BarnettEtal2018a,BarnettEtal2018b,DhamalaEtal:2018}.}, although its symptoms had previously been noted, particularly in the spectral domain \citep[see \eg,][]{DingEtal:2006,Chen:2006}\footnote{But note that the ``block-decomposition'' method presented in \citet{Chen:2006} to address the conundrum---essentially an attempt at constructing a single-regression estimator (see \secref{sec:gcsre} below)---is incorrect \citep{Solo:2016}.}.

\subsubsection{Single-regression estimation} \label{sec:gcsre}

The above issues may, however, be sidestepped. Given a finite-order VAR($p$) model \eqref{eq:varp} (again, we assume that $p$ is known, and discuss infinite-order VAR models and model order selection in \secref{sec:infestmod}), the reduced VAR \eqref{eq:varr} may not be assumed finite-dimensional, but the reduced residuals covariance matrix $\Sigma^\rrop$  will nonetheless be a continuous, \emph{deterministic} function
\begin{equation}
	\Sigma^\rrop = \Sigr(\btheta) \label{eq:sigfun}
\end{equation}
of the finite-dimensional full-model parameters $\btheta = (A,\Sigma)$, with $V(\btheta) = \Sigma_{xx}$ for $\btheta \in \Theta_0$. Given parameters $(A,\Sigma)$, the function $\Sigr(\btheta)$ may be computed numerically to desired precision by spectral factorisation in the frequency domain \citep{Dhamala:2008b,Dhamala:2008a}, spectral factorisation in the time domain domain \citep{Barnett:mvgc:2014} or by a state-space method \citep{Barnett:ssgc:2015,Solo:2016} which devolves to solution of a discrete algebraic Riccati equation (DARE); see \apxref{sec:varpss}. Now from \eqref{eq:popgc} and \eqref{eq:sigfun} the population GC is
\begin{equation}
	F_{\bY\to\bX}(\btheta) = \log\frac{|\Sigr(\btheta)|}{|\Sigma_{xx}|} \label{eq:popgcsr}
\end{equation}
Then given a data sample $\bu = \{\bu_1,\ldots,\bu_N\}$ we need only estimate the full VAR($p$) model \eqref{eq:varp}, to obtain full-model ML parameter estimates $\hbtheta(\bu)$; the reduced model estimate $\hSigma^\rrop(\bu) = \Sigr\big(\hbtheta(\bu)\big)$ is calculated directly from the full-model estimates by one of the techniques described above, yielding the \emph{single-regression Granger causality estimator}
\begin{equation}
	\hF^\sre_{\bY\to\bX}(\bu) = \log\frac{\big|\Sigr\big(\hbtheta(\bu)\big)\big|}{\big|\hSigma_{xx}(\bu)\big|} \label{eq:sampgcsr}
\end{equation}
Since ML parameter estimates are consistent, $\hSigma_{xx}\big(\btheta\big) \pconverge \Sigma_{xx}$, and  by the Continuous Mapping Theorem \citep[CMT;][]{LehmannRomano:2005} $V\big(\hbtheta\big) \pconverge \Sigma^\rrop$. Thus the  SR estimator $\hF^\sre_{\bY\to\bX}(\btheta) = F_{\bY\to\bX}\big(\hbtheta\big)$ is consistent.

\section{Asymptotic null distribution for single-regression GC estimators} \label{sec:asymdist}

As regards statistical inference, LR sample statistics in general satisfy Wilks' Theorem \citep{Wilks:1938}, which implies that they are asymptotically $\chi^2(d)$-distributed under the null hypothesis as sample size $N\to\infty$, with degrees of freedom $d = q n_xn_y$, where $q$ is the model order used for the full and reduced regressions; see \secref{sec:infestmod} for further discussion.

The SR estimator $\hF^\sre_{\bY\to\bX}(\btheta)$, however, is \emph{not} a log-likelihood ratio, so Wilks' Theorem does not apply, and the asymptotic distribution under the null hypothesis \eqref{eq:H0p} has thus far remained unknown\footnote{The VAR single-regression estimator is a special case of the more general state-space GC estimator \citep{Barnett:ssgc:2015,Solo:2016}. \citet{Solo:2016} claims without proof (and, as we shall see, incorrectly), that the asymptotic distribution of this estimator will in general be $\chi^2$ under the null hypothesis. \citet{Barnett:ssgc:2015}, find through extensive simulation that the state-space estimator under the null is well-approximated by a $\Gamma$ distribution; this, we shall see, is well-explained here, at least in the VAR case.}. This is the principal subject of our study.
We shall see that, unlike Wilks' asymptotic null $\chi^2$ distribution, the sampling distribution of the single-regression estimator under the null depends explicitly on the (true) null parameters $\btheta \in \Theta_0$ themselves, raising some awkward questions regarding statistical inference, which we address in \secref{sec:statinf}.

We proceed with a technical result---essentially a multivariate $2$nd-order Delta Method \citep{LehmannRomano:2005}---on which our derivation of the asymptotic distributions for the time-domain and band-limited SR estimators hinges:

\thmgap
\begin{proposition} \label{prop:gchi2}
Let $f(\btheta)$ be a non-negative, twice-differentiable function on a smooth $r$-dimensional manifold $\Theta$ which vanishes identically on the $s$-dimensional hyperplane\footnote{If $\Theta_0$ is a general smooth submanifold of $\Theta$, then \propref{prop:gchi2} still holds around any given $\btheta \in \Theta_0$ modulo an appropriate local transformation of coordinates.} $\Theta_0 \subset \Theta$ specified by $\theta_1 = \ldots = \theta_d = 0$ with $d = r-s$. Then
{\renewcommand{\theenumi}{\alph{enumi}}
\begin{enumerate}
	\item The gradient $\nabla\!f(\btheta)$ is zero for all $\btheta \in \Theta_0$. \label{it:fdel1}
	\item Writing a subscript ``$_0$'' to denote the $d \times d$ upper-left submatrix of an $r \times r$ matrix, for $\btheta \in \Theta_0$ the Hessian  $\Hess(\btheta) = \nabla^2\!f(\btheta)$ takes the form
	\begin{equation}
		\Hess(\btheta) = \begin{bmatrix} \Hess_0(\btheta) & 0 \\ 0 & 0 \end{bmatrix}
	\end{equation}
	with $\Hess_0(\btheta)$ positive-semidefinite.
	\label{it:fdel2}
	\item For $\btheta \in \Theta_0$ let $\bvtheta_N$ be a sequence of $r$-dimensional random vectors with $\sqrt N (\bvtheta_N-\btheta) \dconverge \snormal(\bzero,\Omega)$ as $N \to \infty$; \cf~\eqref{eq:mlpdist}. Then
	\begin{equation}
		N f(\bvtheta_N) \dconverge \chi^2\big(\tfrac12 \Hess_0(\btheta),\Omega_0\big) \ \text{as} \ N \to \infty
	\end{equation}
	\label{it:fdel3}
\end{enumerate}}
\end{proposition}
\thmgap

See \apxref{sec:genxprop} for a proof.

\subsection{The time-domain SR estimator} \label{sec:gcsregdist}

We shall apply \propref{prop:gchi2} with $f(\btheta)$ given by the $F_{\bY\to\bX}(\btheta)$ of \eqref{eq:popgcsr}. Firstly, $F_{\bY\to\bX}(\btheta)$ is non-negative and vanishes on $\Theta_0$ \citep{Geweke:1982}. Below we establish that it is also twice-differentiable (in fact \emph{analytic}), so that by \propref{prop:gchi2} with $\bvtheta_N$ the ML parameter estimate $\hbtheta$ for sample size $N$, we have for any $\btheta \in \Theta_0$
\begin{equation}
	N \hF^\sre_{\bY\to\bX}(\btheta) \dconverge \chi^2\!\bracr{\tfrac12 \Hess_0(\btheta),\pcmat_0(\btheta)} \ \text{as} \ N \to \infty \label{eq:gcsregdist}
\end{equation}
where $\Hess(\btheta)$ is the Hessian of $F_{\bY\to\bX}(\btheta)$ and $\Omega(\btheta)$ the inverse Fisher information matrix for the VAR($p$) model \eqref{eq:varp}  evaluated at $\btheta$. Here the ``$_0$'' subscript denotes a submatrix corresponding to the null-hypothesis variable indices $x = \{1,\ldots,n_x\}$, $y = \{n_x+1,\ldots,n\}$.

To calculate the generalised $\chi^2$ parameters, we thus require firstly the null submatrix $\Omega_0(\btheta)$ of the inverse Fisher information matrix  $\Omega(\btheta)$ for a VAR model specified by $\btheta = (A,\Sigma)$; this is a standard result. Let $\bGamma$ be the autocovariance matrix of \eqref{eq:bGamma}. Considering multi-indices $[k,ij]$ for the regression coefficients $A_{k,ij}$ (so that $k$ indexes lags and $i,j$ variables), the entries for the inverse Fisher information matrix corresponding to the $A_{k,ij}$ are given by \citep{Hamilton:1994,Lutkepohl:1993}
\begin{equation}
	\Omega(\btheta)_{[k,ij][k',i'j']} = \Sigma_{ii'} \IGamma_{kk',jj'} \label{eq:Omgea0}
\end{equation}
where $\IGamma_{kk',jj'}$ denotes the $jj'$ entry of the $kk'$-block of $\iGamma$; or, expressed as a Kronecker product:
\begin{equation}
	 \text{autoregression coefficients block of } \Omega(\btheta)\ = \Sigma \otimes \iGamma
\end{equation}
$\Omega_0(\btheta)$ is then given by the submatrix of \eqref{eq:Omgea0} with $i,i' \in x$ and $j,j' \in y$, which we write as
\begin{equation}
	\Omega_0(\btheta) = \Sigma_{xx} \otimes \bracs{\iGamma}_{yy} \label{eq:finfo0}
\end{equation}

Secondly, to calculate the null Hessian $\Hess_0(\btheta)$, we require an expression for the function $\Sigr(\btheta)$ of \eqref{eq:sigfun}. This we accomplish via the state-space formalism introduced in \citet{Barnett:ssgc:2015}. In \apxref{sec:varpss} we show that
\begin{equation}
	V(\btheta) = A_{xy} \Pi A_{xy}^\trop + \Sigma_{xx} \label{eq:sigred}
\end{equation}
where the $pn_y \times pn_y$ symmetric matrix $\Pi$ is the solution of the discrete algebraic Riccati equation (DARE)
\begin{equation}
	\Pi = \bA_{yy} \Pi \bA_{yy}^\trop + \bSigma_{yy} - \bracr{\bA_{yy} \Pi A_{xy}^\trop + \bSigma_{yx}} \bracr{A_{xy} \Pi A_{xy}^\trop + \Sigma_{xx}}^{-1} \bracr{A_{xy}\Pi \bA_{yy}^\trop + \bSigma_{yx}^\trop} \label{eq:vardare}
\end{equation}
with
\begin{equation}
	\bA_{yy} =
	\begin{bmatrix}
		A_{1,yy} & A_{2,yy} & \ldots & A_{p-1,yy} & A_{p,yy} \\
		I   & 0   & \ldots & 0       & 0   \\
		0   & I   & \ldots & 0       & 0   \\
		\vdots & \vdots & \ddots & \vdots & \vdots   \\
		0   & 0   & \ldots & I       & 0
	\end{bmatrix} \qquad\quad
	A_{xy} =
	\begin{bmatrix}
		A_{1,xy} & A_{2,xy} & \ldots & A_{p-1,xy} & A_{p,xy}
	\end{bmatrix}
\end{equation}
and
\begin{equation}
	\bSigma_{yy} =
	\begin{bmatrix}
		\Sigma_{yy} & 0 & \ldots & 0 \\
		0      & 0 & \ldots & 0 \\
		\vdots & \vdots & \ddots & \vdots \\
		0      & 0 & \ldots & 0
	\end{bmatrix} \qquad\quad
	\bSigma_{yx} =
	\begin{bmatrix}
		\Sigma_{yx} \\ 0 \\ \vdots \\ 0
	\end{bmatrix}
\end{equation}

It is not hard to see that, by construction, $F_{\bY\to\bX}(\btheta) = \log|V(\btheta)| - \log|\Sigma_{xx}|$ is an analytic function of $\btheta$: calculation of $V(\btheta)$ via \eqref{eq:vardare} and \eqref{eq:sigred} only involves algebraic operations (solution of multivariate polynomial equations), and we know $V(\btheta)$ to be positive-definite, so that $|V(\btheta)| > 0\;\forall\btheta$ and $\log|V(\btheta)|$ is thus analytic.

To calculate $\Hess_0(\btheta)$ we require derivatives up to $2$nd order of $V(\btheta)$, as defined implicitly through \eqref{eq:sigred} and \eqref{eq:vardare}, with respect to the null-hypothesis parameters (that is, with respect to $A_{k,ij}$ for $i \in x$, $j \in y$), evaluated for $\btheta \in \Theta_0$. From \eqref{eq:sigred} we may calculate:
\begin{equation}
	\frac{\partial V_{ii'}}{\partial A_{k,uv}}
	 = \delta_{ui} \bracs{\Pi A_{xy}^\trop}_{k,vi'}
	 + \delta_{ui'} \bracs{A_{xy} \Pi}_{k,iv}
	 + \bracs{A_{xy} \frac{\partial \Pi}{\partial A_{k,uv}} A_{xy}^\trop}_{ii'} \label{eq:d1V}
\end{equation}
where indices $i,i',u,u' \in x$, indices $j,j',v,v' \in y$ and $k,k' = 1,\ldots,p$. Since $A_{xy}$ vanishes under the null hypothesis, we have
\begin{equation}
	\left.\frac{\partial V_{ii'}}{\partial A_{k,uv}}\right|_{\btheta \in \Theta_0} = 0 \label{eq:d1V0}
\end{equation}
and from \eqref{eq:d1V} we find
\begin{equation}
	\left.\frac{\partial^2 V_{ii'}}{\partial A_{k,uv} \partial A_{k',u'v'}}\right|_{\btheta \in \Theta_0} = \bracs{\delta_{ui} \delta_{u'i'} + \delta_{ui'} \delta_{u'i}} \Pi_{kk',vv'} \label{eq:d2V0}
\end{equation}
We see then that $\Pi$ is required only on the null space $A_{xy} = 0$, in which case the DARE \eqref{eq:vardare} becomes a DLYAP equation:
\begin{equation}
	\Pi  = \bA_{yy} \Pi \bA_{yy}^\trop + \bSigma_{yy} - \bSigma_{yx} [\Sigma_{xx}]^{-1} \bSigma_{yx}^\trop \label{eq:varlyap}
\end{equation}
We may now calculate the required Hessian. For null parameters $\theta_\alpha,\theta_\beta$, from the definition \eqref{eq:popgcsr} and using \eqref{eq:d1V0} we may calculate
\begin{equation}
	\left.\frac{\partial^2 F_{\bY\to\bX}}{\partial\theta_\alpha \partial\theta_\beta}\right|_{\btheta \in \Theta_0} = \left.\frac{\partial^2 \log|V|}{\partial\theta_\alpha \partial\theta_\beta}\right|_{\btheta \in \Theta_0} = \trace{[\Sigma_{xx}]^{-1} \left.\frac{\partial^2 V}{\partial\theta_\alpha \partial\theta_\beta}\right|_{\btheta \in \Theta_0}} \label{eq:d2FYX}
\end{equation}
\eqref{eq:d2V0} then yields
\begin{equation}
	[\Hess_0(\btheta)]_{[k,uv],[k',u'v']} = 2\bracs{[\Sigma_{xx}]^{-1}}_{uu'} \Pi_{kk',vv'}
\end{equation}
or
\begin{equation}
	\tfrac12 \Hess_0(\btheta) = [\Sigma_{xx}]^{-1} \otimes \Pi \label{eq:2hess0}
\end{equation}

We note that in \eqref{eq:varlyap},
\begin{equation}
	\bSigma_{yy} - \bSigma_{yx} [\Sigma_{xx}]^{-1} \bSigma_{yx}^\trop = \bSigma_{yy|x} =
	\begin{bmatrix}
		\Sigma_{yy|x} & 0 & \ldots & 0 \\
		0      & 0 & \ldots & 0 \\
		\vdots & \vdots & \ddots & \vdots \\
		0      & 0 & \ldots & 0
	\end{bmatrix}
\end{equation}
where $\Sigma_{yy|x} - \Sigma_{yx} [\Sigma_{xx}]^{-1} \Sigma_{xy}$ is a partial covariance matrix. Thus \eqref{eq:varlyap} [\cf~\eqref{eq:gamlyap}] specifies the autocovariance matrix [\cf~\eqref{eq:bGamma}] for a notional $n_y$-dimensional VAR($p$) model with parameters $(A_{yy},\Sigma_{yy|x})$. Accordingly, we write $\Pi$ as $\bGamma_{yy|x}$ from now on, so that \eqref{eq:2hess0} becomes $\tfrac12 \Hess_0(\btheta) = [\Sigma_{xx}]^{-1} \otimes \bGamma_{yy|x}$, and from \eqref{eq:gcsregdist} and \eqref{eq:finfo0} it follows that
\begin{equation}
	N\hF_{\bY\to\bX}^\sre(\btheta) \dconverge \chi^2\bracr{[\Sigma_{xx}]^{-1} \otimes \bGamma_{yy|x}, \Sigma_{xx} \otimes\IGamma_{yy}} \quad\text{as } N \to \infty
\end{equation}
But from the transformation invariance $\chi^2(C A C^\trop,B) = \chi^2(A,C^\trop B C)$ and the mixed-product property of the Kronecker product, we may verify that the $\Sigma_{xx}$ terms cancel ($\Sigma_{xx}$ is positive-definite, and thus has an invertible Cholesky decomposition). We thus state our first principal result as

\thmgap
\begin{theorem} \label{thm:gcsregd}
The asymptotic distribution of the single-regression Granger causality estimator under the null hypothesis $\btheta \in \Theta_0$ (\ie, $A_{k,xy} = 0\, \forall k$) is given by
\begin{equation}
	N\hF^\sre_{\bY\to\bX}(\btheta) \dconverge \chi^2\!\bracr{\bI_{xx} \otimes \bGamma_{yy|x}, \bI_{xx} \otimes \IGamma_{yy}} \quad\text{as } N \to \infty \label{eq:gcsregd}
\end{equation}
where $\bI_{xx}$ is the $n_x \times n_x$ identity matrix, and  $\bGamma$, $\bGamma_{yy|x}$ satisfy the respective DLYAP equations
\begin{subequations}
\begin{align}
	\bGamma -\bA \bGamma \bA^\trop &= \bSigma \label{eq:bGDLYAP} \\
	\bGamma_{yy|x} -\bA_{yy} \bGamma_{yy|x} \bA_{yy}^\trop &= \bSigma_{yy|x} \label{eq:bGyy|xDLYAP}
\end{align} \label{eq:bGDLYAPs}%
\end{subequations}
or, equivalently, the corresponding Yule-Walker equations \eqref{eq:yw}.
\hfill$\blacksquare$
\end{theorem}
\thmgap

From \eqref{eq:gcsregd} we see that the limiting distribution of $N\hF^\sre_{\bY\to\bX}(\btheta)$ is the sum of $n_x$ random variables independently and identically distributed as $\chi^2\!\bracr{\bGamma_{yy|x},\IGamma_{yy}}$. Through \eqref{eq:gchi2}, this distribution may be expressed in terms of the eigenvalues of $\bI_{xx} \otimes \bracr{\IGamma_{yy}\bGamma_{yy|x}}$; these are in fact the eigenvalues $\lambda_1,\ldots,\lambda_{pn_y}$ of $\IGamma_{yy}\bGamma_{yy|x}$, where each $\lambda_i$ appears with multiplicity $n_x$. The asymptotic distribution of $N\hF^\sre_{\bY\to\bX}(\btheta)$ under the null thus takes the form of a weighted sum of $pn_y$ iid $\chi^2(n_x)$ variables:
\begin{equation}
	\lambda_1 W_1 + \ldots + \lambda_{pn_y} W_{pn_y}\,, \qquad W_i \text{ iid} \sim \chi^2(n_x)
\end{equation}
The asymptotic mean and variance of $N\hF^\sre_{\bY\to\bX}(\btheta)$ for $\btheta \in \Theta_0$ are
\begin{subequations}
\begin{align}
	\expect{N\hF^\sre_{\bY\to\bX}(\btheta)} &\to n_x\sum_{i=1}^{pn_y} \lambda_i    \label{eq:srgcmean} \\
	\var{N\hF^\sre_{\bY\to\bX}(\btheta)}    &\to 2n_x\sum_{i=1}^{pn_y} \lambda_i^2 \label{eq:srgcvar}
\end{align} \label{eq:srgcmeanvar}%
\end{subequations}
respectively, from which the $\Gamma$-approximation of the generalised $\chi^2$ distribution may be obtained, via \eqref{eqs:gc2meanvar} and \eqref{eqs:gamparms}. \figref{fig:gcsreg_dist_1} plots generalised $\chi^2$, $\Gamma$-approximation and empirical SR-estimator cumulative density functions (CDFs) for a representative null VAR model with $n_x = 3$, $n_y = 5$ and $p = 7$ for several sample sizes, illustrating asymptotic convergence with increasing sample length $N$. The $\Gamma$-approximation is barely distinguishable from the generalised $\chi^2$.
\begin{figure}
	\begin{center}
	\includegraphics{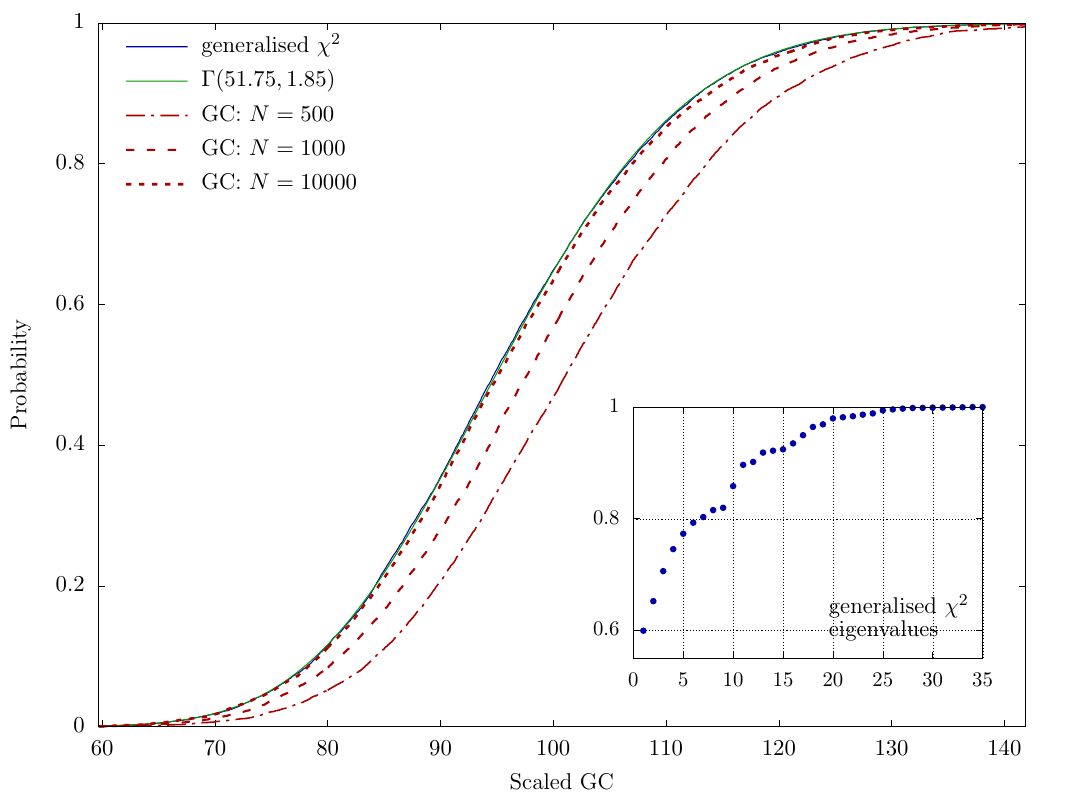}
	\end{center}
	\caption{Generalised $\chi^2$, $\Gamma$-approximation and empirical SR-estimator CDFs for a representative null VAR model with $n_x = 3$, $n_y = 5$ and $p = 7$, for several sample sizes. The null VAR model was randomly generated according to the scheme described in \apxref{sec:ranvar}, with spectral radius $\rho = 0.9$ and residuals generalised correlation $\gamma = 1$. The generalised $\chi^2$ and $\Gamma$-approximation are almost indistinguishable. Empirical GC plots were based on $10^4$ generated time series for each sample lengths $N$. Inset figure: the  $pn_y = 35$ distinct eigenvalues for the generalised $\chi^2$ distribution, sorted by size. (Each eigenvalue will be repeated $n_x = 3$ times.)} \label{fig:gcsreg_dist_1}
\end{figure}

We note that the eigenvalues $\lambda_i$ are all $>0$, since both $\IGamma_{yy}$ and $\bGamma_{yy|x}$ are guaranteed to be positive-definite (this follows from the purely nondeterministic assumption on the process $\bU_t$). From \eqref{eq:gamparmsa} and \eqref{eq:srgcmeanvar}, we find that the shape parameter of the $\Gamma$-approximation satisfies\footnote{This follows from the identities $\bracr{\sum_{i=1}^{pn_y} \lambda_i}^2 = \sum_{i=1}^{pn_y} \lambda_i^2 + \sum_{i,j=1}^{pn_y} \lambda_i \lambda_j = pn_y \sum_{i=1}^{pn_y} \lambda_i^2 - \frac12 \sum_{i,j=1}^{pn_y} (\lambda_i - \lambda_j)^2$.}
\begin{equation}
	\frac{n_x}2 \le \alpha \le \frac{pn_xn_y}2\,,
\end{equation}
and $\alpha = \frac{pn_xn_y}2 \iff$ all the $\lambda_i$ are equal $\iff$ $p = n_y = 1$, in which case the distribution of $N\hF^\sre_{\bY\to\bX}(\btheta)$ is asymptotically $\chi^2(n_x)$ scaled by $\lambda$ (\cf~\secref{sec:examp}). We also state the following conjecture, which we have tested extensively empirically, but have so far been unable to prove rigorously:

\thmgap
\begin{conjecture} \label{cnj:lambda1}
The eigenvalues of $\IGamma_{yy}\bGamma_{yy|x}$ satisfy $\lambda_i \le 1$ for $i = 1,\ldots,pn_y$.
\end{conjecture}
\thmgap

If \cnjref{cnj:lambda1} holds, then from \eqref{eq:gamparmsb} and \eqref{eq:srgcmeanvar} the scale parameter of the $\Gamma$-approximation satisfies:
\begin{equation}
	0 \le \beta \le 2
\end{equation}
Simulations reveal that spectral radius and residuals generalised correlation of null parameters $\btheta$ have a strong effect on the distribution of the eigenvalues $\lambda_i$. Spectral radius close to $1$ and large residuals correlation give rise to a larger spread of eigenvalues $< 1$, resulting in asymptotic null sampling distributions significantly different from a (non-generalised) $\chi^2$. \figref{fig:genchi2_vs_chi2} presents the distribution of single-regression estimator CDFs under random sampling of VAR models of given size, spectral radius and residuals generalised correlation (see \apxref{sec:ranvar} for the VAR sampling method), where the effects of spectral radius and residuals generalised correlation on the null distribution via the eigenvalues (inset figures) is clearly seen.
\begin{figure}
	\begin{center}
	\includegraphics{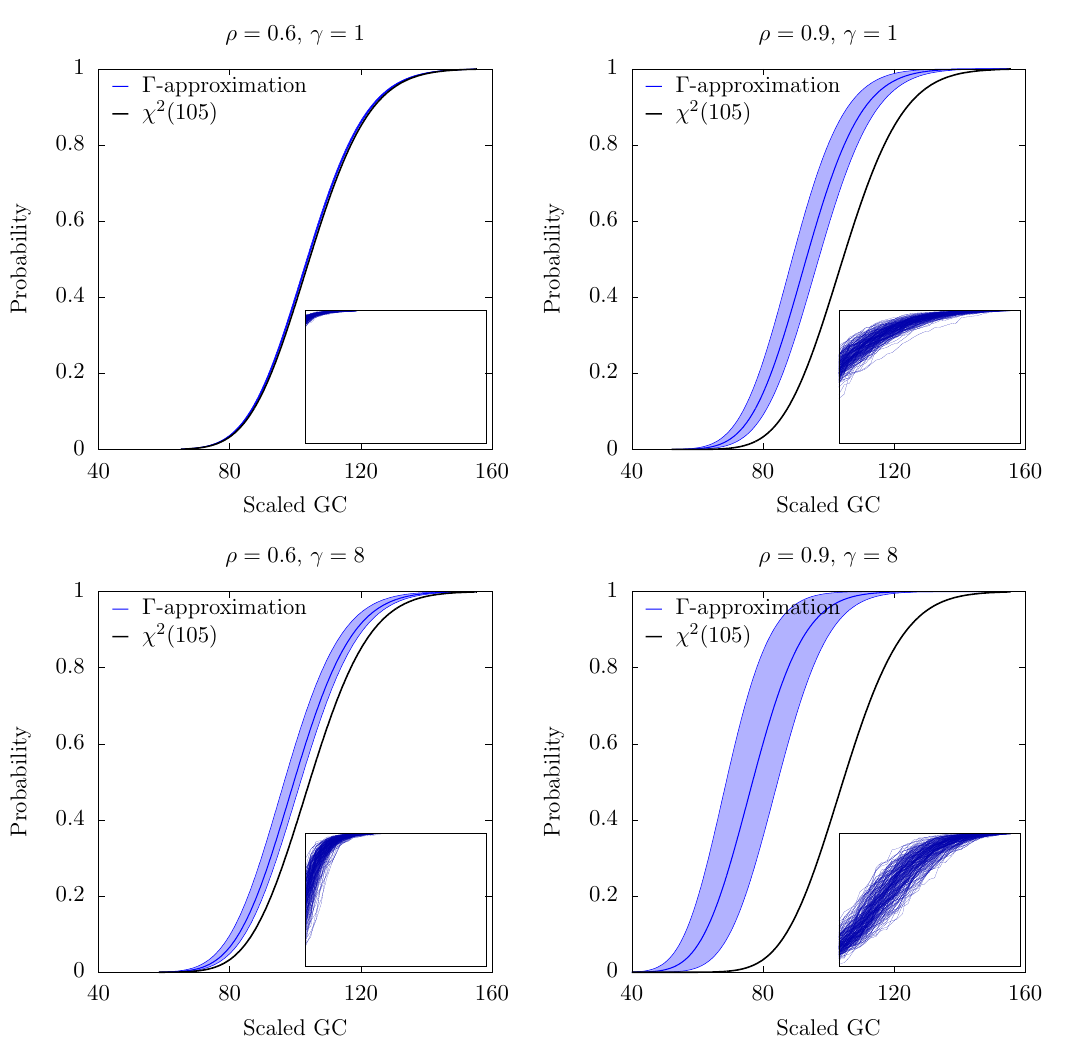}
	\end{center}
	\caption{Distribution of $\Gamma$-approximation CDFs for a random sample of $200$ null VAR models (\apxref{sec:ranvar}) with $n_x = 3$, $n_y = 5$ and $p = 7$, for a selection of spectral radii $\rho$ and residuals generalised correlation $\gamma$. At each scaled GC value, blue lines plot the mean of the $\Gamma$-approximation CDFs, while shaded areas bound upper/lower $95\%$ quantiles. Black lines plot the corresponding LR $\chi^2$ CDFs, with $pn_xn_y = 105$ degrees of freedom.  Inset figures: the  $pn_y = 35$ distinct eigenvalues sorted by size, for each of the $200$ generalised $\chi^2$ distributions ($x$ range is $1-35$, $y$ range is $0-1$).} \label{fig:genchi2_vs_chi2}
\end{figure}

Assuming \cnjref{cnj:lambda1}, an immediate consequence of \eqref{eq:srgcmeanvar} and $N\hF^\lre_{\bY\to\bX} \dconverge \chi^2(pn_xn_y)$ is that for $\btheta \in \Theta_0$
\begin{subequations}
\begin{align}
	\expect{N\hF^\sre_{\bY\to\bX}(\btheta)} &\le \expect{N\hF^\lre_{\bY\to\bX}(\btheta)} \\
	\var{N\hF^\sre_{\bY\to\bX}(\btheta)}    &\le \var{N\hF^\lre_{\bY\to\bX}(\btheta)}
\end{align}
\end{subequations}
This suggests that, more generally, the SR estimator will have smaller bias and better efficiency than the LR estimator\footnote{Strictly speaking we require the distributions of the estimators in the \emph{non}-null case to draw this conclusion; but see also \secref{sec:infestmod}.}. As is apparent in \figref{fig:genchi2_vs_chi2}, the reduction in bias and variance is stronger for spectral radius close to $1$ and high residuals correlation.

\subsection{The band-limited spectral SR estimator} \label{sec:sgc}

We now consider the asymptotic null-distribution of the band-limited spectral Granger Causality estimator $\hat{f}_{\bY\to\bX}(\frange;\btheta)$ \eqref{eq:sgcbl} which facilitates inference on the corresponding band-limited population statistic $f_{\bY\to\bX}(\frange;\btheta)$. As explained below, the appropriate null hypothesis in this case is in fact identical to the time-domain (or ``broadband'') null condition $H_0$. Inference on $f_{\bY\to\bX}(\frange;\btheta)$ is nevertheless informative beyond a time-domain test; thus while we may reject $H_0$ in the neighbourhood of $\omega_1$ at some significance level, we may fail to reject $H_0$ at the same level around a different frequency $\omega_2$, with the implication that while $H_0$ likely does not hold, Granger causality is likely to be significant around $\omega_1$ but negligible around $\omega_2$. In other words, while the time-domain test is sensitive to \textit{any} (sufficiently large) deviation from the null hypothesis---regardless of localisation in the frequency spectrum---the band-limited test is sensitive to discrepancies within the specified frequency band, but insensitive to discrepancies outside this band. Thus a significant band-limited test can justifiably be attributed to a sizeable contribution from the frequency band in question, whereas a significant time-domain test allows only to draw the less-specific conclusion that contributions from across the entire frequency range are potentially implicated in the result.

As we shall see in the following, the limiting asymptotic distribution under $H_0$ of the band-limited estimator $\hf_{\bY\to\bX}([\omega-\eps,\omega+\eps];\btheta)$ as $\eps \to 0$ exists. It should \emph{not}, however, be conflated with the asymptotic distribution of $\hf_{\bY\to\bX}(\omega;\btheta)$ under the \textit{point-frequency null hypothesis} $H_0(\omega)$ (see \secref{sec:extend} for discussion of the point-frequency case).

The point-frequency spectral GC $f_{\bY\to\bX}(\omega;\btheta)$ is non-negative, and for any $\omega$ clearly vanishes under $H_0$. We note further that by assumed stability of the VAR($p$) \eqref{eq:varp}, the inverse transfer function $\Phi(\omega)$ does not vanish, so that $\Psi(\omega)$---and hence, via \eqref{eq:specfac}, $S(\omega)$ and consequently $f_{\bY\to\bX}(\omega;\btheta)$---are analytic functions of the phase angle $\omega$ [as well as of the $\btheta = (A,\Sigma)$]. For a frequency range $\frange \subseteq [0,2\pi]$ with measure $|\frange| > 0$, then, $f_{\bY\to\bX}(\frange;\btheta)$ vanishes iff $f_{\bY\to\bX}(\omega;\btheta)$ is identically zero; \ie, precisely under the original null hypothesis $H_0$. Given a frequency range $\frange$, we thus apply \propref{prop:gchi2} to the asymptotic distribution of the band-limited spectral Granger causality estimator $\hf_{\bY\to\bX}(\frange;\btheta)$ under the (appropriate) original null hypothesis \eqref{eq:H0p} $H_0 : A_{k,xy} = 0\;\forall k$.

In the previous section we calculated the covariance $\Omega_0(\btheta) = \Sigma_{xx} \otimes\IGamma_{yy}$ of null parameters under $H_0$; it remains to calculate the null Hessian $\Hess_0(\frange;\btheta)$. Since (Lebesgue) integration and partial differentiation are linear operations, it follows that the Hessian for $f_{\bY\to\bX}(\frange;\btheta)$---on the original null space $\Theta_0$---is just
\begin{equation}
	\Hess_0(\frange;\btheta) = \frac1{|\frange|} \int_\frange \Hess_0(\omega;\btheta)\,d\omega
\end{equation}
where, for given $\omega$, $\Hess_0(\omega;\btheta)$ is the Hessian of $f_{\bY\to\bX}(\omega;\btheta)$ on $\Theta_0$ with respect to the original null parameters $A_{k,xy}$, $k = 1\ldots,p$.

Dropping the ``$\omega$'' and ``$\btheta$'' arguments for compactness where convenient, on the null space $\Theta_0$ we have  $\Phi_{xy} = 0$, and since then $\Phi$ is lower block-triangular, we have also $\Psi_{xx} = [\Phi_{xx}]^{-1}$, $\Psi_{yy} = [\Phi_{yy}]^{-1}$, and  $\Psi_{xy} = 0$. The CPSD for the process $\bX_t$ is given by
\begin{equation}
	S_{xx} = [\Psi S \Psi^*]_{xx} =  \Psi_{xx} \Sigma_{xx} \Psi_{xx}^* + \Psi_{xy} \Sigma_{yx} \Psi_{xx}^* + \Psi_{xx} \Sigma_{xy} \Psi_{xy}^* + \Psi_{xy} \Sigma_{yy} \Psi_{xy}^*
\end{equation}
On the null space $S_{xx} = \Psi_{xx} \Sigma_{xx} \Psi_{xx}^*$ so that $[S_{xx}]^{-1} = \Phi^*_{xx} [\Sigma_{xx}]^{-1} \Phi_{xx}$,

We define $T(\omega)$ to be the $n_x \times n_x$ (Hermitian) matrix
\begin{equation}
	T(\omega) = \Psi_{xy}(\omega) \Sigma_{yy|x} \Psi_{xy}(\omega)^*
\end{equation}
so that from \eqref{eq:sgc} $f_{\bY\to\bX}(\omega) = \log|S_{xx}(\omega)| - \log|S_{xx}(\omega)-T(\omega)|$. $T(\omega)$ vanishes on the null space. We may check that (for $p,q,r,s = 1,\ldots,n$, $k = 1,\ldots,p)$
\begin{equation}
	\frac{\partial \Psi_{pq}}{\partial A_{k,rs}} = \Psi_{pr} \Psi_{sq} e^{-i\omega k} \label{eq:dPsi}
\end{equation}
from which we may calculate  (with $i,i',u,u' \in x$, $j,j',v,v' \in y$)
\begin{equation}
	\frac{\partial T_{ii'}}{\partial A_{k,uv}}
	= \sum_{j,j'} [\Sigma_{yy|x}]_{jj'} \bracr{\Psi_{iu} \Psi_{vj} \bar\Psi_{i'j'} e^{-i\omega k} + \bar\Psi_{i'u} \bar\Psi_{vj'} \Psi_{ij} e^{i\omega k}} \label{eq:d1T}
\end{equation}
so that in particular, $\left.\displaystyle \frac{\partial T}{\partial \theta_\alpha}\right|_{\btheta \in \Theta_0} = 0$ for a null parameter $\theta_\alpha$, and we find [\cf~\eqref{eq:d2FYX}]:
\begin{equation}
	\left.\frac{\partial^2 f_{\bY\to\bX}}{\partial \theta_\alpha \partial \theta_\beta}\right|_{\btheta \in \Theta_0} =
	\trace{[S_{xx}]^{-1} \left.\frac{\partial^2 T}{\partial \theta_\alpha \partial \theta_\beta}\right|_{\btheta \in \Theta_0}} \label{eq:d2fYX}
\end{equation}
for null parameters $\theta_\alpha, \theta_\beta$. From \eqref{eq:d1T} and using \eqref{eq:dPsi}, we may calculate
\begin{equation}
	\left.\frac{\partial^2 T_{ii'}}{\partial A_{k,uv} \partial A_{k',u'v'}}\right|_{\btheta \in \Theta_0}
	= \Psi_{iu} \Psi^*_{u'i'} \bracs{S_{yy|x}}_{vv'} e^{-i\omega(k-k')} + \Psi_{iu'} \Psi^*_{ui'} \bracs{S_{yy|x}}_{v'v} e^{i\omega(k-k')} \label{eq:d2T}
\end{equation}
where
\begin{equation}
	S_{yy|x} = \Psi_{yy} \Sigma_{yy|x} \Psi^*_{yy} \label{eq:Syy|x}
\end{equation}
is the CPSD for a VAR($p$) model with parameters $(A_{yy},\Sigma_{yy|x})$ [\cf~\secref{sec:gc}]. From \eqref{eq:d2fYX} and \eqref{eq:d2T} we find
\begin{align}
	\left.\frac{\partial^2 f_{\bY\to\bX}}{\partial A_{k,uv} \partial A_{k',u'v'}}\right|_{\btheta \in \Theta_0}
	&= \bracs{[\Sigma_{xx}]^{-1}}_{uu'} \bracc{\bracs{S_{yy|x}}_{vv'} e^{-i\omega(k-k')} + \bracs{S_{yy|x}}_{v'v}  e^{i\omega(k-k')}} \\
	&= \bracs{[\Sigma_{xx}]^{-1}}_{uu'} \bracs{S_{yy|x} e^{-i\omega(k-k')} + \bar S_{yy|x}  e^{i\omega(k-k')}}_{vv'} \quad\text{since } S_{yy|x} \text{ is Hermitian} \\
	&= 2\bracs{[\Sigma_{xx}]^{-1}}_{uu'} \real{\bracs{\bS_{yy|x}}_{kk',vv'}}
\end{align}
where for given $\omega$ we define the $pn_y \times pn_y$ Hermitian matrix
\begin{equation}
	\bracs{\bS_{yy|x}(\omega)}_{kk',vv'} = [S_{yy|x}(\omega)]_{vv'} e^{-i\omega(k-k')} \label{eq:SSyy|x1}
\end{equation}
or
\begin{equation}
	\bS_{yy|x}(\omega) = \bZ \otimes S_{yy|x}(\omega)\,, \qquad \bZ_{kk'} = e^{-i\omega(k-k')} \label{eq:SSyy|x}
\end{equation}
We note that $\bS_{yy|x}(\omega)$ is the CPSD for the ``companion VAR($1$)'' [\cf~\eqref{eq:vartrick}] of the VAR($p$) model with parameters $(A_{yy},\Sigma_{yy|x})$, and as such may be thought of as the spectral counterpart of the autocovariance matrix $\bGamma_{yy|x}$ of \secref{sec:gcsregdist}. We thus have, for $\omega \in [0,2\pi]$:
\begin{equation}
	\Hess_0(\omega;\btheta) = [\Sigma_{xx}]^{-1} \otimes \real{\bS_{yy|x}(\omega)}
\end{equation}
so that
\begin{equation}
	\Hess_0(\frange;\btheta) = [\Sigma_{xx}]^{-1} \otimes \real{\bS_{yy|x}(\frange)}
\end{equation}
with
\begin{equation}
	\bS_{yy|x}(\frange) = \frac1{|\frange|} \int_\frange \bS_{yy|x}(\omega) \,d\omega \label{eq:SSyy|xf}
\end{equation}
We may thus state

\thmgap
\begin{theorem} \label{thm:sgcsregd}
The asymptotic distribution of the single-regression band-limited Granger causality estimator over a frequency range $\frange \subseteq [0,2\pi]$ under the null hypothesis $H_0 : A_{k,xy} = 0\, \forall k$, is given by
\begin{equation}
	N\hf_{\bY\to\bX}(\frange;\btheta) \dconverge \chi^2\!\bracr{\bI_{xx} \otimes \real{\bS_{yy|x}(\frange)}, \bI_{xx} \otimes \IGamma_{yy}} \quad\text{as } N \to \infty \label{eq:sgcsregd}
\end{equation}
where $\bS_{yy|x}(\frange)$ is given by \eqref{eq:Syy|x}, \eqref{eq:SSyy|x} and \eqref{eq:SSyy|xf}, and $\bGamma$ is as in \thrmref{thm:gcsregd}.
\hfill$\blacksquare$
\end{theorem}
\thmgap

In particular, for $\frange = [0,2\pi]$ we have
\begin{align*}
	[\bS_{yy|x}([0,2\pi])]_{kk',vv'}
	&= \frac1{2\pi} \int_{\omega = 0}^{2\pi} [\bS_{yy|x}(\omega)]_{kk',vv'}\,d\omega \\
	&= \frac1{2\pi} \int_{\omega = 0}^{2\pi} [S_{yy|x}(\omega)]_{vv'} e^{-i\omega(k-k')} d\omega \\
	&= [\Gamma_{yy|x}]_{k'-k,vv'} \qqquad\text{by \eqref{eq:acov}} \\
	&= [\bGamma_{yy|x}]_{kk',vv'}
\end{align*}
so that
\begin{equation}
	\bS_{yy|x}([0,2\pi]) = \bGamma_{yy|x}
\end{equation}
Thus we confirm that the distribution of $\hf_{\bY\to\bX}(\frange;\btheta)$ is consistent with \thrmref{thm:gcsregd} and \eqref{sgcint} for $\frange = [0,2\pi]$. We note that the limiting asymptotic distribution of $\hf_{\bY\to\bX}([\omega-\eps,\omega+\eps];\btheta)$ as $\eps \to 0$ is obtained by simply replacing $\bS_{yy|x}(\frange)$ by $\bS_{yy|x}(\omega)$ in \eqref{eq:sgcsregd}; as mentioned, this should not be confused with the distribution of the point-frequency estimator $\hf_{\bY\to\bX}\big(\omega;\btheta\big)$ under the null hypothesis \eqref{eq:H0z}.

As before, the generalised $\chi^2$ distribution in \eqref{eq:sgcsregd} may be described in terms of the eigenvalues $\lambda_i$ of $\IGamma_{yy}\bS_{yy|x}(\frange)$. For $p > 2$, since the null space of the null hypothesis $H_0$ used to derive \eqref{eq:sgcsregd} has dimension $pn_xn_y$---whereas for any given angular frequency $\omega$ the dimension of the point-frequency null $H_0(\omega)$ (\secref{sec:extend}) is $2n_xn_y$---only $2n_y$ of the $pn_y$ eigenvalues of $\IGamma_{yy}\bS_{yy|x}(\omega)$ will be non-zero. In contrast to the time-domain case, the eigenvalues of $\IGamma_{yy}\bS_{yy|x}(\frange)$ will not necessarily be $\le 1$ (\cf~\cnjref{cnj:lambda1}), but empirically we observe that the maximum eigenvalue shrinks to $1$ as the bandwidth $|\frange|$ increases to $2\pi$.

\subsection{Worked example: the general bivariate \texorpdfstring{VAR($1$)}{VAR(1)}} \label{sec:examp}

In \apxref{sec:gcbivar1} we analyse the bivariate VAR($1$)
\begin{subequations}
\begin{align}
	X_t &= a_{xx} X_{t-1} + a_{xy} Y_{t-1} + \eps_{xt} \\
	Y_t &= a_{yx} X_{t-1} + a_{yy} Y_{t-1} + \eps_{yt}
\end{align} \label{eq:bivar1}%
\end{subequations}
with residuals covariance matrix $\expect{\beps_t \beps_t^\trop} = \twomat{\sigma_{xx}}{\sigma_{xy}}{\sigma_{yx}}{\sigma_{yy}}$, so that $\btheta = (a_{xx},a_{xy},a_{yx},a_{yy},\sigma_{xx},\sigma_{xy},\sigma_{yy})$. We calculate that the (population) Granger causality from $Y \to X$ is given by
\begin{equation}
	F_{Y \to X}(\btheta) = \log\frac{\sigr(\btheta)}{\sigma_{xx}}
\end{equation}
where the reduced residuals covariance matrix---in this case the scalar $\Sigr(\btheta) = \sigr(\btheta)$---is given by
\begin{equation}
	\sigr(\btheta) = \frac12 \bracr{P + \sqrt{P^2-Q^2}}
\end{equation}
with
\begin{equation}
	P = \sigma_{xx} (1+a_{yy}^2) - 2\sigma_{xy} a_{xy} a_{yy} + \sigma_{yy} a_{xy}^2\,, \qquad
	Q = 2(\sigma_{xx} a_{yy} - \sigma_{xy} a_{xy})
\end{equation}

We apply \thrmref{thm:gcsregd} to calculate the asymptotic distribution of the SR estimator $\hF^\sre_{Y \to X}(\btheta)$ on the null space $a_{xy} = 0$. Noting that for model order $p = 1$,  $\bGamma = \Gamma_0$, and setting $\Gamma^{-1}_0 = [\omega_{ij}]$, we have $[\bGamma^{-1}]_{yy} = [\Gamma_0^{-1}]_{yy} = \omega_{yy}$ [\apxref{sec:gcbivar1}, eq.~\eqref{eq:omyy}]. The DLYAP equation \eqref{eq:bGyy|xDLYAP} for $\Gamma_{yy|x}$ is just
\begin{equation}
	\Gamma_{yy|x} - a_{yy}^2 \Gamma_{yy|x} = \sigma_{yy|x}
\end{equation}
with
\begin{equation}
	\sigma_{yy|x} = \sigma_{yy} - \frac{\sigma_{xy}^2}{\sigma_{xx}} = \sigma_{yy}\bracr{1-\kappa^2}
\end{equation}
where $\displaystyle \kappa = \frac{\sigma_{xy}}{\sqrt{\sigma_{xx}\sigma_{yy}}}$ is the residuals correlation, so that
\begin{equation}
	\Gamma_{yy|x} = \bracr{1-\kappa^2}\frac{\sigma_{yy}}{1-a_{yy}^2}
\end{equation}
and the single eigenvalue of $[\bGamma^{-1}]_{yy}\bGamma_{yy|x}$ is just
\begin{equation}
	\lambda =  \bracr{1-\kappa^2}\frac{\sigma_{yy}\omega_{yy}}{1-a_{yy}^2}
\end{equation}
By \thrmref{thm:gcsregd} the asymptotic distribution of the single-regression estimator is thus a scaled $\chi^2(1)$:
\begin{equation}
	N \hF^\sre_{Y \to X}(\btheta) \dconverge \lambda \cdot \chi^2(1) = \Gamma\bracr{\frac{1}{2},2\lambda}
\end{equation}
(the $\Gamma$-approximation in this case is exact).

In \apxref{sec:gcbivar1} we calculate the spectral Granger causality from $Y \to X$ at $\omega \in [0,2\pi]$ as
\begin{equation}
	f_{Y \to X}(\omega;\btheta) = \log\frac{P-Q\cos\omega}{P-Q\cos\omega - a_{xy}^2 \sigma_{yy|x}}
\end{equation}
We find then that
\begin{equation}
	S_{yy|x}(\omega) = \sigma_{yy|x} |\psi_{yy}(\omega)|^2 = \sigma_{yy}\bracr{1-\kappa^2} \frac{|1-a_{xx} z|^2}{|\Delta(\omega)|^2}
\end{equation}
with $\Delta(\omega) = |\Phi(\omega)|$. On the null space, $\Delta(\omega) = (1-a_{xx}z)(1-a_{yy}z)$, so that
\begin{equation}
	S_{yy|x}(\omega) = \bracr{1-\kappa^2}\frac{\sigma_{yy}}{|1-a_{yy}z|^2} = \bracr{1-\kappa^2}\frac{\sigma_{yy}}{1-2a_{yy}\cos\omega + a_{yy}^2} \label{eq:Syy|xbivar}
\end{equation}
In this case, since the model order is $p = 1$, the null hypothesis \eqref{eq:H0z} coincides with the null hypothesis \eqref{eq:H0p} (\ie, $a_{xy} = 0$), so that from \thrmref{thm:sgcsregd} we have
\begin{equation}
	N \hf_{Y \to X}(\omega;\btheta) \dconverge \lambda(\omega) \cdot \chi^2(1)
\end{equation}
where
\begin{equation}
	\lambda(\omega) = \bracr{1-\kappa^2}\frac{\sigma_{yy}\omega_{yy}}{|1-a_{yy}z|^2} = \bracr{1-\kappa^2}\frac{\sigma_{yy}\omega_{yy}}{1-2a_{yy}\cos\omega + a_{yy}^2}
\end{equation}
The asymptotic distribution for the band-limited estimator may then be calculated as per \eqref{eq:SSyy|xf} by integrating \eqref{eq:Syy|xbivar} across the appropriate frequency range\footnote{We may calculate that $\displaystyle \int \!\frac{d\omega}{1-2a\cos\omega+a^2} = \frac2{1-a^2} \tan^{-1}\!\bracr{\frac{1+a}{1-a}\tan\frac\omega2}$.}.

\section{Statistical inference with the single-regression estimators} \label{sec:statinf}

In this section we show how our principal results, Theorems~\ref{thm:gcsregd} and \ref{thm:sgcsregd}, can be usefully applied to inference on Granger causality $F_{\bX\to\bY}$. The key difficulty in this endeavour is the fact that the asymptotic distribution of the single-regression test statistics depends explicitly on the true null parameter, i.e. on its location in the subset of the parameter space associated with the null hypothesis.

Assuming the VAR($p$) class of models, and given time-series data $\bu = \{u_1,u_2,\ldots,u_N\}$, we can estimate a ML parameter $\hbtheta(\bu)$; but we cannot simply replace the VAR($p$) true null parameter for the asymptotic distribution by the ML estimate, since $\hbtheta(\bu)$  cannot be assumed to lie in the null space $\Theta_0$. Our solution is to project the ML parameter estimate onto the null space by a continuous projection $\Pi: \Theta \to \Theta_0$ (so that in particular $\Pi\cdot\btheta = \btheta$ if $\btheta \in \Theta_0$). Here we might choose the obvious projection $\Pi\cdot\btheta = (0,\ldots,0,\theta_{d+1},\ldots,\theta_r)$. This yields a Neyman-Pearson test for the null hypothesis of zero Granger causality from $\bY\to\bX$, which we term a ``Projection Test''. Below we show that the Projection Test is ``asymptotically valid'', in the sense that the Type I error rate (incidence of false rejections of $H_0$) may be limited to a prespecified level $\alpha$, and investigate statistical power of the SR estimators in terms of the Type II error rate (incidence of failures to correctly reject $H_0$) at a given level.

For brevity, in what follows $F(\cdots)$ will refer to either the time-domain population SR statistic $F^\sre_{\bX\to\bY}(\cdots)$ \eqref{eq:popgcsr} or the band-limited population statistic $f_{\bX\to\bY}(\frange;\cdots)$ \eqref{eq:sgcbl} for some given frequency range $\frange$.

\subsection{Type I error rate} \label{sec:typeI}

Given a null parameter $\btheta \in \Theta_0$ under $H_0$ \eqref{eq:H0p} for a VAR($p$) model, let $\Phi_\btheta(\cdots)$ be the CDF of the corresponding asymptotic null Granger causality estimator; \ie, the generalised $\chi^2$ of \eqref{eq:gcsregd} (time domain) or \eqref{eq:sgcsregd} (band-limited). Given a data sample $\bu$ of size $N$, the order-$p$ Projection Test procedure at level $\alpha$ is as follows:
\begin{enumerate}
	\item Calculate the ML VAR($p$) parameter estimate $\hbtheta(\bu)$.
	\item Calculate the GC estimate $F\big(\hbtheta(\bu)\big)$.
	\item Reject the null hypothesis if
	\begin{equation}
		\Phi_{\Pi\cdot\hbtheta(\bu)}\bracr{NF\big(\hbtheta(\bu)\big)} > 1-\alpha \label{eq:stest}
	\end{equation}
\end{enumerate}
We now investigate the asymptotic Type I error rate for this test. Let $\hbtheta$ be the ML parameter estimator for a model with true parameter $\btheta \in \Theta_0$; recall that we consider $\hbtheta$ as a random variable parametrised by $\btheta$, with implicit dependence on sample size $N$, and we note again that $\hbtheta$ may take values outside of $\Theta_0$. The GC estimator is $F\big(\hbtheta\big)$ (again, a sample size-dependent random variable parametrised by $\btheta$). Given $\btheta \in \Theta_0$, the probability of a Type I error for the Projection Test is
\begin{equation}
	P_I(\btheta;\alpha) = \prob{\Phi_{\Pi\cdot\hbtheta}\bracr{NF\big(\hbtheta\big)} > 1-\alpha} \label{eq:typeI}
\end{equation}
We wish to show that $\lim_{N \to \infty} P_I(\btheta;\alpha) = \alpha$.

\thmgap
\begin{lemma} \label{thm:slemma}
	Suppose given a sequence of pairs of real-valued random variables $(X_n,Y_n)$ such that
	\begin{align}
		X_n &\dconverge X \\
		Y_n &\pconverge c
	\end{align}
	where $c$ is a constant. Then
	\begin{equation}
		\prob{X_n \le Y_n} \converge \prob{X \le c} \label{eq:slemma}
	\end{equation}
\end{lemma}
\begin{proof}
	By Slutsky's Lemma \citep{LehmannRomano:2005}, we have $X_n-Y_n \dconverge X-c$. Thus $\forall\,\eps > 0$, $\exists\,n_0 \in {\mathbb N}$ such that $\forall\,n \ge n_0$ we have $\big|\prob{X_n-Y_n \le 0}- \prob{X-c \le 0}\big| < \eps$, or, equivalently $\big|\prob{X_n \le Y_n}- \prob{X \le c}\big| < \eps$, which establishes \eqref{eq:slemma}.
\end{proof}
\thmgap

\noindent Now \eqref{eq:typeI} is equivalent to
\begin{equation}
	P_I(\btheta;\alpha) = 1-\prob{NF\big(\hbtheta\big) \le \Phi^{-1}_{\Pi\cdot\hbtheta}(1-\alpha)}
\end{equation}
Since $\hbtheta$ is a consistent estimator for $\btheta$ and $\Pi$ is continuous with $\Pi\cdot\btheta = \btheta$, by the Continuous Mapping Theorem (CMT) we have $\Pi\cdot\hbtheta \pconverge \btheta$. It is not hard to verify that the inverse CDF $\Phi^{-1}_\btheta(\cdots)$ evaluated at $1-\alpha$ is continuous in the $\btheta$ argument, so that again by the CMT we have $\Phi^{-1}_{\Pi\cdot\hbtheta}(1-\alpha) \pconverge \Phi^{-1}_{\btheta}(1-\alpha)$. By \thrmref{thm:gcsregd}, $NF\big(\hbtheta\big) \dconverge Q_\btheta$, where $Q_\btheta$ is a (generalised $\chi^2$) random variable with CDF $\Phi_\btheta$. Applying \lemmref{thm:slemma} to the pair-sequence $\bracr{NF\big(\hbtheta\big), \Phi^{-1}_{\Pi\cdot\hbtheta}(1-\alpha)}$ we have
\begin{equation}
	P_I(\btheta;\alpha)
	\converge 1-\prob{Q_\btheta  \le \Phi^{-1}_\btheta(1-\alpha)} = 1-\Phi_\btheta\bracr{\Phi^{-1}_\btheta(1-\alpha)} = \alpha \label{eq:typeIvalid}
\end{equation}
as required. While the rate of convergence of $P_I(\btheta;\alpha)$ to $\alpha$ can be expected to depend on the true null parameter $\btheta$, empirically we find that the effect is small\footnote{The extent to which the choice of projection $\Pi$ may affect the rate of convergence is unclear. It is plausible that orthogonal projection with respect to the Fisher metric \citep{Amari:2016} evaluated at the null parameter may lead to faster convergence, but we have not tested this hypothesis.}.

To gain insight into comparative Type I error rates and their parameter dependencies, we performed the following experiment: at each data length $N$, we drew $1,000$ model samples $\btheta$ from a distribution $\tbtheta$ over the subspace $\Theta(\rho,\gamma) \subset \Theta$ of null VAR models of given size, for a given spectral radius $\rho$ and residuals generalised correlation $\gamma$, according to the scheme described in \apxref{sec:ranvar}. For each model $\btheta$, we generate $n = 1,000$ data sequences. For each data sequence we calculate LR and SR estimates of the Granger causality, and test the estimates against the appropriate $\chi^2$ (resp. generalised $\chi^2$) asymptotic null distribution at significance level $\alpha = 0.05$. For a given VAR model $\btheta$, let $p(\btheta)$ be the population Type I error rate associated with the given data length, estimator, inference method and significance level. Each of the $n$ statistical tests on a given model is a Bernoulli trial, yielding a binomially-distributed estimator $\hp(\btheta)$ for $p(\btheta)$; \ie, $n\hp(\btheta) \sim B\big(n,p(\btheta)\big)$:
\begin{equation}
	\prob{n\hp(\btheta) = k} = \binom nk p(\btheta)^k [1-p(\btheta)]^{n-k}\,, \qquad k = 0,\ldots,n
\end{equation}
The $\hp(\btheta)$ estimates are collated over model samples from $\tbtheta$, to yields a distribution $\hp$ for each data length $N$; explicitly,
\begin{equation}
	\cprob{n\hp = k}{\tbtheta = \btheta} = \prob{n\hp(\btheta) = k}, \qquad k = 0,\ldots,n
\end{equation}
Results for $n_x = 3$, $n_y = 5$, $p = 7$, $\gamma = 1$, $N = 2^8 - 2^{14}$ and a selection of spectral radii are presented in \figref{fig:Type_I}. We see that asymptotic convergence for the SR estimator Type I error rate is slightly better on average than for the LR estimator, especially for smaller sample lengths and spectral radius $\rho(A)$ \eqref{eq:specrad} close to $1$.
\begin{figure}
	\begin{center}
	\includegraphics{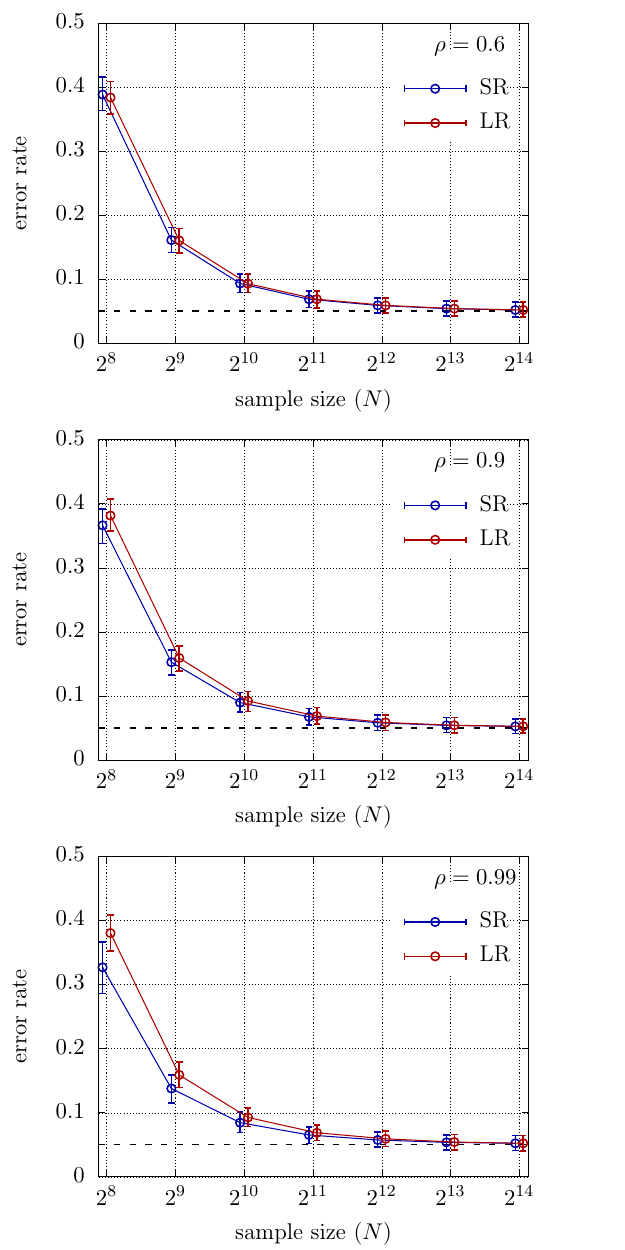}
	\end{center}
	\caption{Type I error rates at significance level $\alpha = 0.05$ for the SR vs. LR estimator for null VAR models with $n_x = 3$, $n_y = 5$, $p = 7$, $\gamma = 1$ and spectral radius $\rho$ as indicated, plotted against data length $N$ (note log scale). Error bars represent $95\%$ upper/lower quantiles; the dashed horizontal line indicates the significance level. See text (\secref{sec:typeI}) for simulation details.} \label{fig:Type_I}
\end{figure}

Note that there are two sources of variation in the experiment: estimation error of the per-model error rate, represented by the estimators $\hp(\btheta)$, and variation across the distribution $\tbtheta$ of models.
From the Law of Total Variance, we may calculate that
\begin{equation}
	\var{\hp} = \varwrt\tbtheta{p(\tbtheta)} + \frac1n \expectwrt\tbtheta{p(\tbtheta)[1-p(\tbtheta)]}
\end{equation}
We note that the second term is $\le \frac1{4n}$, so that for large enough $n$ the $\varwrt\tbtheta{p(\tbtheta)}$ term---the variance of Type II error rate with respect to model variation---dominates; \ie, for large enough $n$ the error bars are essentially due to the effect of model variation on the GC estimates. In \figref{fig:Type_I}, where the contribution to the error bars from the per-model rate estimators $\hp(\btheta)$ is small (of the order\footnote{$95\%$ quantiles correspond approximately to two standard deviations.} $\frac1{\sqrt n} \approx 0.032$), the visible dispersion is thus mostly due to model variation $\tbtheta$.

\subsection{Type II error rate (statistical power)} \label{sec:typeII}

As regards statistical power, the Type II error rate given a \emph{non}-null parameter $\btheta \in \Theta$, is given by
\begin{equation}
	P_{II}(\btheta;\alpha) = \prob{\Phi_{\Pi\cdot\hbtheta}\bracr{NF\big(\hbtheta\big)} \le 1-\alpha} \label{eq:typeII}
\end{equation}
For the likelihood-ratio statistic we have the classical result due to \citet{Wald:1943}, which yields that the scaled estimator is $\approx$ non-central $\chi^2(pn_xn_y;NF)$ in the large-sample limit, where $F$ is the population GC. The approximation only holds with reasonable accuracy for ``small'' values of $F$. In the single-regression case we have no equivalent result (but see discussion in \secref{sec:disc}). Clearly, $P_{II}(\btheta;\alpha)$ will depend strongly on the population GC value associated with specific parameters $\btheta$, but, as for the null case, will still vary within the subspace of parameters which yield a given population statistic; that is, for given $F > 0$, $P_{II}(\btheta;\alpha)$ will vary over the set $\{\btheta: F(\btheta) = F\}$. We can at least say that, roughly
\begin{align}
	P_{II}(\btheta;\alpha) &\approx 0 \quad\text{for}\quad N \gg \frac{\Phi_\btheta^{-1}(1-\alpha)}{F(\btheta)} \\[0.4em]
	P_{II}(\btheta;\alpha) &\approx 1 \quad\text{for}\quad N \ll \frac{\Phi_\btheta^{-1}(1-\alpha)}{F(\btheta)}
\end{align}

To compare statistical power (\ie, $1-$Type II error rate) between the SR and LR estimators, and to examine their dependencies on model parameters, we performed simulations on the bivariate VAR($1$) of \secref{sec:examp}. Results are displayed in \figref{fig:bivar_Type_II}.
\begin{figure}
	\begin{center}
	\includegraphics{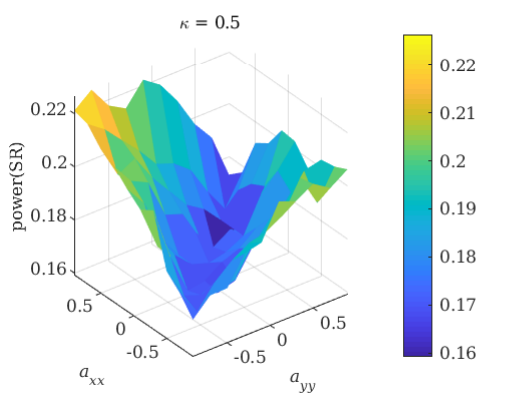}
	\includegraphics{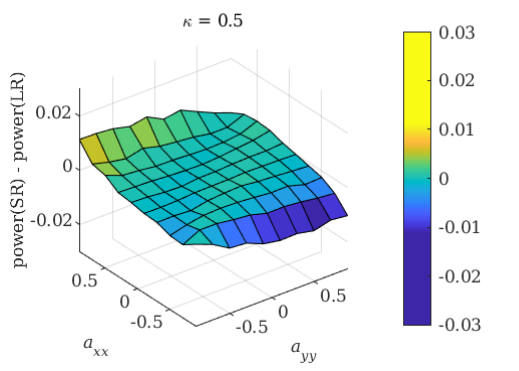} \\[1em]
	\includegraphics{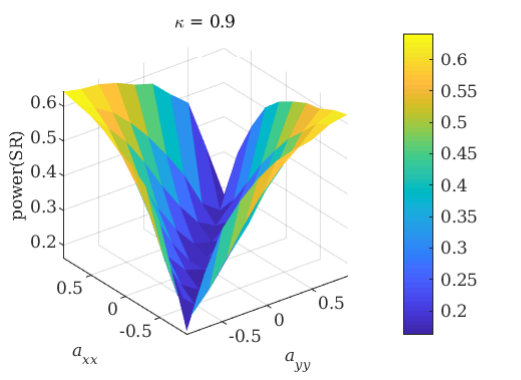}
	\includegraphics{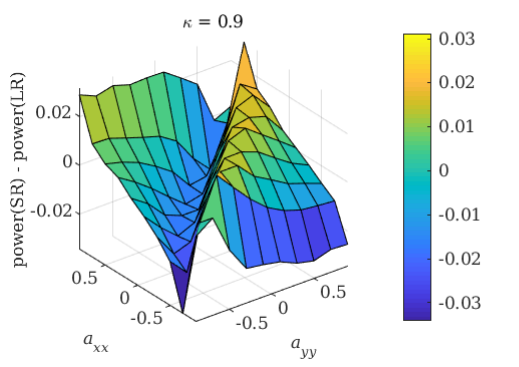}
	\end{center}
	\caption{Statistical power ($1-\text{Type II error rate}$) for the bivariate VAR($1$) (\ref{eq:bivar1}, \secref{sec:examp}), plotted against null-space parameters $a_{xx},a_{yy}$ for residuals correlation $\kappa = 0.5$ and $\kappa = 0.9$. Other parameters were $N = 10^4$, $a_{yx} = 0$, $F = 10^{-4}$. Results were based on $10^4$ simulated VAR($1$) sequences for each set of null-space parameters. Left column: statistical power for the SR estimator. Right column: difference in statistical power (\SRE~ power $-$ \LRE~power).} \label{fig:bivar_Type_II}
\end{figure}
We see (left-column figures) that for the SR estimator, statistical power is roughly symmetrical around $a_{xx} = a_{yy}$ (we note that in this case, spectral radius is $\max\{|a_{xx}|,|a_{yy}|\}$, so spectral radius ``swaps'' between $a_{xx}$ and $a_{yy}$ along this line). Statistical power is highest when $a_{xx}$ and $a_{yy}$ are roughly equal in magnitude and opposite in sign. Parameter dependency is stronger for larger residuals correlation $\kappa$. The power difference (right-column figures) between the estimators is small when residuals correlation is small; for larger residuals correlation, there is again rough (anti-)symmetry in the power difference around $a_{xx} = a_{yy}$, but the dependencies are quite complex.

To gain some insight into comparative statistical power on more general VAR models, we also repeated the experiment described in \secref{sec:typeI}, but this time generating non-null models with population GC $F = 0.007$ (see \apxref{sec:ranvar}), and testing for Type II errors. Results are presented in \figref{fig:Type_II}.
\begin{figure}
	\begin{center}
	\includegraphics{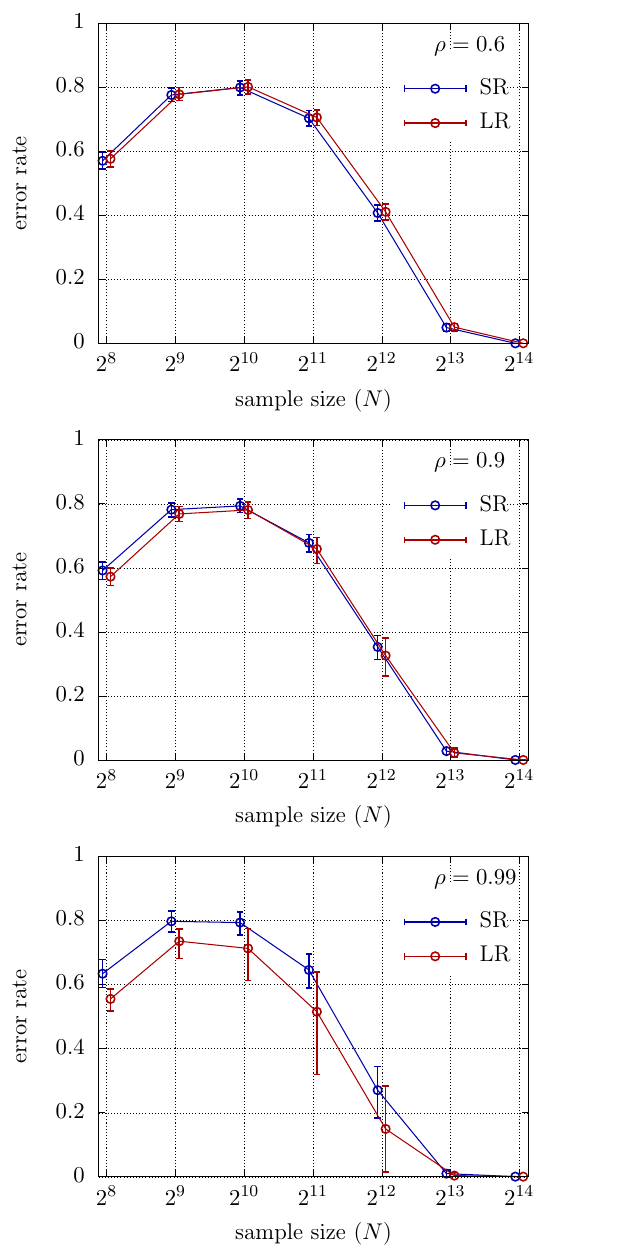}
	\end{center}
	\caption{Type II error rates ($1 - \text{statistical power}$) at $\alpha = 0.05$ for the SR vs. LR estimator for causal VAR models with $n_x = 3$, $n_y = 5$, $p = 7$,  $\gamma = 1$, population GC $F = 0.007$, and spectral radius $\rho$ as indicated; other details as in \figref{fig:Type_I}.} \label{fig:Type_II}
\end{figure}
The overall conclusion is that the SR and LR tests have similar statistical power, except when spectral radius is very close to $1$, where the LR test has somewhat higher power. Again, visible dispersion is mostly due to model variation $\tbtheta$.

\figref{fig:Type_II} raises a puzzling point: we might suspect that the distribution of $\hp$ for the LR statistic has \emph{small} variance due to $\tbtheta$, since---at least in the null case---the asymptotic distribution depends only on the problem size (degrees of freedom). This is not what our simulations show, however (see in particular $\rho = 0.99$ and $2^{10} \le N \le 2^{12}$). But note that in the non-null case this may not apply; if we are in the ``Wald regime'', where $F$ is ``small'', then similar should apply, since the Wald non-central $\chi^2$ non-null asymptotic distribution again depends only on degrees of freedom (and the actual GC value $F$). We can only conclude that for the particular parameters of the experiment, $F$ lies beyond the ambit of Wald's Theorem, so that somewhat counter-intuitively---that is, despite the known dependence of the SR sampling distribution in the null case on VAR parameters---variance of Type II error rate is generally \emph{smaller} than for the LR case.

\section{Discussion} \label{sec:disc}

In this study we have derived, for the first time, asymptotic generalised $\chi^2$ sampling distributions for the (unconditional) single-regression Granger causality in the time domain, and for the band-limited single-regression estimator in the frequency domain. We conclude with some discussion on implications, limitations and future extensions of this work.

\subsection{Unknown and infinite VAR model order} \label{sec:infestmod}

So far we have assumed that the model order of the underlying VAR model is both finite and known. However, these restrictions will generally not be met in practice. The question, then, is how statistical inference is affected when the true model order is unknown, infinite or both. Apart from some general remarks (see below), we consider here only the case of finite, but unknown VAR model order $p$.

If the model order is unknown, statistical inference becomes a two-stage process: first we obtain a parsimonious model order estimate $\hp$ using a standard selection procedure \citep{McQuarrie:1998}. We then compute a test-statistic $F^{(\hp)}$---the log-likelihood ratio or single-regression estimate---using the selected model order; that is, maximum-likelihood VAR($\hp)$ parameter estimates are computed for the full (and in the case of the LR estimator, reduced) models.

The central question is how the model order selection step should be performed so that the two-step testing procedure as a whole is (1) asymptotically valid, and (2) is as statistically powerful as possible. We consider first the implications of selecting a fixed model order $q\neq p$ for inference on a VAR($p$) process. If $q<p$, then the asymptotic statements underlying the likelihood-ratio test (i.e. Wilks' theorem) and projection test (\thrmref{thm:gcsregd} above) no longer hold; thus, for instance, in the case of the LR statistic, $N F^{(q)} \dconverge \chi^2(qn_x n_y)$ fails under the null hypothesis \eqref{eq:H0p}. On the other hand, if $q > p$ then Wilks' theorem \textit{does} apply, because (\cf~\secref{sec:gcmle}) one can always subsume a VAR model by a higher-order model. Thus, under the null hypothesis, the asymptotic statement $N F^{(q)} \dconverge \chi^2(qn_x n_y)$ is correct for fixed $q>p$, even though the model is over-specified. However, simulation  results suggest two problems associated with using too-large a model order: firstly, the rate of convergence of the test statistic decreases with $q$---potentially leading to inflated Type I errors in small samples---and secondly, statistical power is degraded.

These considerations imply that \textit{for the purposes of statistical inference}, model order $\hp$ should be selected for the \textit{full}, rather than reduced process. Because the reduced process will in general be infinite order, the reduced model order estimate will diverge to infinity as sample size increases, leading to suboptimal inference as described above. For estimation of Granger causality \emph{as an effect size}, however, \textit{efficiency} of the estimate takes precedence. For the SR estimator there is again no reason not to select for model order based on the full process. For the LR estimator, although for fixed sample size $N$ choice of full vs. reduced model order selection runs into the bias-variance trade-off identified by \citet{StokesPurdon:2017} (\secref{sec:gcmle}), empirical investigation reveals that, if model order $\hp$ is selected for the reduced process by a standard scheme, both bias and variance tend to zero as $N\to\infty$. This is explained by the observation that $\hp$ grows comparatively slowly with increasing $N$. Indeed, if the population Granger causality $F > 0$ is small enough that the large-sample non-central $\chi^2$ approximation of \citet{Wald:1943} is viable for the LR estimator, both bias and variance are $\bigO{\hp N^{-1}}$, while in general $\hp$ will grow more slowly than $\bigO N$ \citep[Chap. 5]{HandD:2012}. Further research is required here, though ultimately the issue is of little import since under the effect-size scenario there seems scant reason \emph{not} to prefer the SR estimator on grounds of efficiency; \cf~the closing remarks in \secref{sec:gcsregdist}.

Given that model order is selected based on the full process, there is still a choice to be made regarding which of the many possible selection criteria should be deployed. In this regard we note that if the model order selection criterion utilised is \textit{consistent}\footnote{Note that the popular Akaike Information Criterion (AIC) is \textit{not} consistent, whereas, \eg, Schwartz' Bayesian Information Criterion (BIC) and  Hannan amd Quinn's Information Criterion  \citep[HQIC;][]{HandQ:1979} are consistent.} then the probability of choosing the correct model order converges to 1 as $N\to\infty$. In this case the asymptotic results underlying the likelihood-ratio test and single regression projection test remain valid even if the model order has to be estimated:
\begin{align*}
	\lim_{N\to\infty}\prob{F^{(\hp)} \le \eps}
	&= \lim_{N\to\infty} \bracr{\prob{\hp = p} \cprob{F^{(\hp)} \le \eps}{\hp = p} + \prob{\hp \ne p} \cprob{F^{(\hp)} \le \eps}{\hp \ne p}} \\
	&= \lim_{N\to\infty} \prob{\hp = p} \cprob{F^{(\hp)} \le \eps}{\hp = p} + \lim_{N\to\infty} \prob{\hp \ne p} \cprob{F^{(\hp)} \le \eps}{\hp \ne p} \\
	&= \lim_{N\to\infty} \prob{\hp = p} \cprob{F^{(\hp)} \le \eps}{\hp = p} + 0 \\
	&= \lim_{N\to\infty}  \cprob{F^{(\hp)} \le \eps}{\hp = p}
\end{align*}
We have shown (\thrmref{thm:gcsregd} and \secref{sec:typeI}) that for the Projection Test the final limit is the CDF of the corresponding generalised $\chi^2$ distribution evaluated at $\eps$, while for the LR test, by Wilks' theorem it is the CDF of the corresponding $\chi^2$ distribution evaluated at $\eps$. Hence, the convergence statements justifying these tests still hold, and the corresponding two-step procedures including the model selection step are asymptotically valid.

As regards the infinite-order VAR case, establishing an asymptotically-valid scheme would seem to be more difficult, and merits further research. Preliminary experiments indicate that, at least for VARMA (equivalently state-space) processes, the (autoregressive) SR estimator with consistent model order selection yields asymptotically valid inference, in the sense that the Type I error rate converges to a specified significance level $\alpha$ (\cf~\secref{sec:typeI}), and is also more efficient than the LR estimator. This seems reasonable, given the possibility (see \secref{sec:extend} below) of extending the analysis in this study to the state-space GC estimator \citep{Barnett:ssgc:2015,Solo:2016}.

\subsection{Extensions and future research directions} \label{sec:extend}

We make the observation here that \emph{any} estimator of the form $f\big(\hbtheta\big)$, where $\hbtheta$ is the ML parameter estimator, will converge in distribution to a generalised $\chi^2$ distribution under the associated null hypothesis $\btheta \in \Theta_0$ if the statistic $f(\btheta)$ satisfies the prerequisites for \propref{prop:gchi2}. This covers a range of extensions to our results, which we list below. They vary in tractability according to the difficulty of explicit calculation of the Fisher information matrix $\Omega_0(\btheta)$ and Hessian $\Hess_0(\btheta)$ for $\btheta \in \Theta_0$.
\begin{itemize}

\item\label{it:extend:cond} \textbf{The conditional case}

Extending the time-domain \thrmref{thm:gcsregd} to the conditional case \citep{Geweke:1984} is reasonably straightforward. Given a partitioning $\bU_t = [\bX_t^\trop\;\bY_t^\trop\;\bZ_t^\trop]^\trop$ of the variables, the time-domain conditional population Granger causality statistic is given by
\begin{equation}
	F_{\bY\to\bX|\bZ}(\btheta) = \log\frac{\big|\Sigma^\rrop_{xx}\big|}{|\Sigma_{xx}|} \label{eq:cpopgc}
\end{equation}
where now the reduced regression is
\begin{align}
	\bX_t &= \sum_{k=1}^\infty A^\rrop_{k,xx}\bX_{t-k} + A^\rrop_{k,xz}\bZ_{t-k} + \beps^\rrop_{x,t} \label{eq:cvarrx} \\
	\bZ_t &= \sum_{k=1}^\infty A^\rrop_{k,zx}\bX_{t-k} + A^\rrop_{k,zz}\bZ_{t-k} + \beps^\rrop_{z,t} \label{eq:cvarrz}
\end{align}
with residuals covariance matrix $\Sigma^\rrop = \expect{\beps^\rrop_t \beps_t^{\rrop\trop}} = \begin{bmatrix} \Sigma^\rrop_{xx} & \Sigma^\rrop_{xz} \\[0.5em] \Sigma^\rrop_{zx} & \Sigma^\rrop_{zz}  \end{bmatrix}$. Again, $\Sigma^\rrop$ is a deterministic (albeit more complicated) function $V(\btheta)$ of the VAR parameters, which may again be expressed in terms of a DARE \citep{Barnett:ssgc:2015}. Although more complex, derivation of the appropriate Hessian proceeds along the same lines as in \secref{sec:gcsregdist}.

Extension to the conditional case in the frequency domain (band-limited estimator) is substantially more challenging, due to the complexity of the statistic; see \eg, \citet[Sec.~II]{Barnett:ssgc:2015}. Note that while the unconditional spectral statistic only references the full model parameters, in the conditional case both full and reduced model parameters are required.

\item\label{it:extend:spectral} \textbf{The spectral point-frequency estimator}

The null hypothesis $H_0(\omega)$ for vanishing of $f_{\bY\to\bX}(\omega;\btheta)$ \eqref{eq:sgc} at the point frequency $\omega$ is $\Psi_{xy}(\omega) = 0$, where $\Psi(\omega)$ is the transfer function \eqref{eq:trfun} for the VAR model, or \citep{Geweke:1982}
\begin{equation}
H_0(\omega) : \left\{
\begin{array}{r}
	\displaystyle \sum_{k=1}^p A_{k,xy} \cos k\omega = 0 \\[1em]
	\displaystyle \sum_{k=1}^p A_{k,xy} \sin k\omega = 0
\end{array}
\right. \label{eq:H0z}
\end{equation}
For given $\omega$, \eqref{eq:H0z} represents $2n_xn_y$ constraints on the $pn_xn_y$ regression coefficient matrices\footnote{So if $\omega \ne k\pi$ for any $k$, then if $p \le 2$ we have the original $H_0 : A_{k,xy} = 0$; \cf~\secref{sec:examp}.} $A_k$. Calculation of the point-frequency asymptotic sampling distribution is in principle approachable via a similar technique as before (\propref{prop:gchi2} must be adjusted for the case of more general linear constraints). However, we contend that in real-world applications it makes more sense in any case to consider inference on spectral Granger causality on a (possibly narrow-band) frequency range via the band-limited spectral GC  $f_{\bY\to\bX}(\frange;\btheta)$ \eqref{eq:sgcbl} as discussed in \secref{sec:sgc} rather than at point frequencies. Firstly, for a given VAR($p$), if the ``broadband'' null condition $H_0$ is not satisfied, then the point-frequency null condition $H_0(\omega)$ will only be satisfied precisely at most at a finite number of (in practice unknown) point frequencies. Secondly, real-world spectral phenomena are likely to be to some extent broadband [\eg, power spectra of neural processes \citep{MitraBokil:2008}] and/or otherwise blurred by noise.

The asymptotic sampling distribution of the point-frequency estimator is nonetheless at least of academic interest, since as far as the authors are aware, it remains unknown. Calculation of the point-frequency sampling distribution under the null hypothesis $H_0(\omega)$ is by no means intractable (at least in the unconditional case), although the computation of the Hessian is considerably complicated by the more restrictive null condition \eqref{eq:H0z}. The condition is, at least, linear, so may be readily transformed to conform to the preconditions for \propref{prop:gchi2}.

\item\label{it:extend:Fstat} \textbf{The single-regression $F$-test statistic}

Granger causality may be tested using a standard $F$-test for vanishing of the appropriate regression coefficients \citep{Hamilton:1994}. The population test statistic is
\begin{equation}
	\frac{\trace{\Sigma^\rrop}-\trace{\Sigma_{xx}}}{\trace{\Sigma_{xx}}}\,, \label{eq:popgcF}
\end{equation}
and the corresponding dual-regression estimator scaled by $d_2/d_1$ is asymptotically $F(d_1,d_2)$-distributed under the null hypothesis, where $d_1 = pn_xn_y$  and $d_2 = n_x[N-p(n_x+n_y+1)]$ are the respective degrees of freedom. Anecdotally, \citep[see, \eg,][]{Hamilton:1994,Lutkepohl:2005}, and especially for smaller sample lengths, this may yield a more powerful test than the (dual-regression) LR test. It would thus be of interest to investigate whether the single-regression $F$-statistic form \eqref{eq:popgcF} obtained by setting $\Sigma^\rrop = \Sigr(\btheta)$ as in \eqref{eq:sigfun} similarly yields a more powerful test than the single-regression Geweke form \eqref{eq:popgcsr}. This statistic satisfies the conditions for \propref{prop:gchi2}, and the  Hessian is as easily calculated as in \secref{sec:gcsregdist}.

We remark that the statistic \eqref{eq:popgcF} is less useful as a quantitative measure of causal effect, as it lacks the information-theoretic interpretation of the Geweke form \eqref{eq:popgc} \citep{Barnett:tegc:2009,Barnett:teml:2012}, nor is it invariant under as broad a set of transformations \citep{Barrett:2010}, or under invertible digital filtering \citep{Barnett:gcfilt:2011}.

\item\label{it:extend:SS} \textbf{The state-space Granger causality statistic}

The state-space GC statistic, unconditional and conditional, in time and frequency domains, was introduced in \citet{Barnett:ssgc:2015} \citep[see also][]{Solo:2016}, and is calculated from the innovations-form state-space parameters $\btheta = (A,C,K,\Sigma)$, for which reduced-model parameters $A^\rrop$ and $C^\rrop$ are known, while $K^\rrop$ and $\Sigma^\rrop$ appear as solutions of a certain DARE (\cf~\secref{sec:gcsregdist} and \apxref{sec:varpss}). The state-space approach extends GC estimation and inference from the class of finite-order VAR models to the super-class of state-space (equivalently VARMA) models. The power of the method derives from the fact that (i) unlike for the class of finite-order VAR models, the class of state-space models is closed under subprocess extraction (an essential ingredient of the Granger causality construct), and (ii) many real-world data, notably econometric and neurophysiological, have a strong moving-average component, and are thus more parsimoniously represented as VARMA rather than pure VAR models. The class of state-space models is in addition---again in contrast to the finite order VAR class---closed under sub-sampling, temporal/spatial aggregation, observation noise, and digital filtering -- all common features of real-world data acquisition and observation procedures.

The single-regression GC statistic as defined here for a VAR($p$) model is essentially the state-space GC statistic for a state-space model which describes the same stochastic process (\secref{sec:gcsregdist} and \apxref{sec:varpss}). The state-space estimator may be defined in similar terms to the VAR case, \ie, based on ML estimates for the model parameters. In comparison with the VAR case, there are two challenges to derivation of the asymptotic null sampling distribution for the state-space GC estimator via the $2$nd-order Delta Method (\propref{prop:gchi2}): (i) calculation of the Fisher information (\ie, distribution of the ML parameter estimators), and (ii) non-linearity of the null condition \citep[eq.~17]{Barnett:ssgc:2015}. While (ii) complicates calculation of the Hessian, (i) is likely to present a more formidable obstacle, due to the considerable complexity of a closed-form expression for the Fisher information matrix \citep{KleinEtal:2000}.

\end{itemize}

\subsection{The alternative hypothesis} \label{sec:H1}

We may consider two approaches to approximating the sampling distribution of the time-domain and band-limited spectral estimators, which address two distinct scenarios.

In the first scenario, we suppose given a fixed true parameter $\btheta \notin \Theta_0$, and consider the asymptotic sampling distribution of the GC statistic as sample length $N\to\infty$. In this case, the preconditions of \propref{prop:gchi2} certainly do not apply; in particular, the gradient of the statistic will not in general vanish at $\btheta$, so that a $1$st-order multivariate Delta Method is appropriate. This yields a normal distribution for the estimator, with mean equal to the population GC. If the statistic is $f\big(\hbtheta\big)$, then explicitly we have
\begin{equation}
	\sqrt N \bracs{f\big(\hbtheta\big) -f(\btheta)} \dconverge \snormal\!\bracr{0,\nabla\!f(\btheta)\cdot\Omega(\btheta)\cdot\nabla\!f(\btheta)^\trop}
\end{equation}
The variance
\begin{equation}
	\sigma^2 = \nabla\!f(\btheta)\cdot\Omega(\btheta)\cdot\nabla\!f(\btheta)^\trop
\end{equation}
may be computed from the known form of the statistic, although the gradients are harder to calculate, since (1) in the time domain the DARE does not, as in the null case (\secref{sec:gcsregdist}) collapse to a DLYAP \eqref{eq:varlyap}, while (2) in the spectral band-limited case (\secref{sec:sgc}), the transfer function $\Psi(\omega;\btheta)$ is no longer block-triangular. Gradients, furthermore, must be calculated with respect to all (rather than just null) parameters.

This scenario is more pertinent in a realistic empirical situation where, for instance, we are reasonably confident (via a Projection Test as described in \secref{sec:typeI}) that an estimated GC is significantly different from zero, and we would like to put confidence bounds on the estimate; it addresses the \textit{efficiency} of the estimator.

Under the second---and more difficult to analyse---scenario, we suppose that sample length $N$ is fixed (but large), and we consider the limiting distribution of the single-regression GC estimator as the true non-null parameter approaches the null subspace $\Theta_0$. We are now in the territory of \citet{Wald:1943}, where the asymptotics of the Taylor expansion (on which the $1$st- and $2$nd-order Delta Methods are based) become a ``balancing act'' between sample length $N$ and the distance between the true parameter and the null subspace. This is likely to be difficult to calculate; we conjecture that (by analogy with Wald's Theorem) the asymptotic distribution will be a non-central generalised $\chi^2$. This scenario is more pertinent to analysis of the \textit{power} of the statistic.

\subsection{Concluding remarks} \label{sec:conc}

In this study we analysed the hitherto unknown large-sample behaviour of single-regression Granger causality estimators. As quantitative measures of causal effect/information transfer, these estimators clearly outperform their (problematic) classical dual-regression likelihood-ratio counterparts, through consistency and reduced bias and variance. Studying their asymptotic null distributions is therefore of importance, as it admits the construction of novel and tailored hypothesis tests. We have shown the distributions to be generalised $\chi^2$ in both the time-domain and band-limited spectral case, and obtained the distributional parameters in readily-calculated form. Based on these results, we introduced a novel asymptotically-valid ``Projection Test'' of the null hypothesis of vanishing causality. This test should prove especially useful for testing Granger causality on specified frequency bands, a commonplace desideratum in various fields of application, including neuroimaging and econometric time-series analysis. We showed that the approach remains valid even if the model order is unknown, provided a consistent model order selection criterion is used. Our approach may be extended to the conditional case, and, potentially, beyond pure autoregressive modelling to causal inference based on the state-space Granger causality estimator. Finally, we outlined an approach to evaluation of the sampling distribution under the alternative hypothesis, knowledge of which would be useful for the analysis of statistical power and construction of confidence bounds. These extensions merit further detailed investigation, being highly relevant to a wide range of practical applications.

\section*{Acknowledgements} \label{sec:ack}

The authors would like to thank the Dr. Mortimer and Theresa Sackler Foundation, which supports the Sackler Centre for Consciousness Science. We are also grateful to Anil K. Seth and Michael Wibral for useful comments on this work.

\subsection*{Author contributions} \label{sec:contrib}

Authors Gutknecht and Barnett contributed equally to this work.

\vspace{4em}

\appendix

\noindent\textbf{\Large Appendices}

\section{Proof of \propref{prop:gchi2}} \label{sec:genxprop}

Let $\btheta = [\bx^\trop\; \by^\trop]^\trop$ where $x_i = \theta_i$, $i = 1,\ldots,d$ and $y_j = \theta_{d+j}$, $j = 1,\ldots,s$. Since by definition $f(0,\by) = 0\;\forall \by$, we have immediately $\nabla_\by f(0,\by) = 0\;\forall \by$. Treating $\by$ as fixed, we expand $f(\bx,\by)$ in a Taylor series around $\bx = 0$:
\begin{equation}
	f(\bx,\by) = \nabla_\bx f(0,\by) \bx + \tfrac12 \bx^\trop \nabla^2_{\bx\bx} f(0,\by) \bx + \tfrac12 \bx^\trop K(\bx,\by) \bx \label{eq:fxy}
\end{equation}
where for fixed $\by$, $K(\bx,\by)$ is a $d \times d$ matrix function of $\bx$, and $\lim_{\bx \to 0} \vnorm{K(\bx,\by)} = 0$. Now we show that since $f(\bx,\by)$ is non-negative, we must have $\nabla_\bx f(0,\by) = 0\; \forall\by$. Suppose, say, $\nabla_{x_1} f(0,\by) = -g < 0$. Setting $x_1 = \eps > 0$  [if $\nabla_{x_1} f(0,\by) > 0$ we take $x_1 = -\eps$] and  $x_2 = \ldots = x_d = 0$, \eqref{eq:fxy} yields
\begin{equation}
	f(\bx,\by) = -g \eps + \tfrac12 \bracs{\nabla^2_{x_1x_1} f(0,\by) + K_{11}(\bx,\by)} \eps^2
\end{equation}
Now since $\lim_{\eps \to 0} \vnorm{K(\bx,\by)} = 0$, we can always choose $\eps$ small enough that
\begin{equation}
	\tfrac12 \bracs{\nabla^2_{x_1x_1} f(0,\by) + K_{11}(\bx,\by)} \eps < g
\end{equation}
so that $f(\eps,0,...,0,\by) < 0$, a contradiction. Thus we have $\nabla_\bx f(0,\by) = 0\; \forall \by$, proving \ref{prop:gchi2}\ref{it:fdel1}.

From \eqref{eq:fxy} we thus have
\begin{equation}
	f(\bx,\by) = \tfrac12 \bx^\trop \nabla^2_{\bx\bx} f(0,\by) \bx + \tfrac12 \bx^\trop K(\bx,\by) \bx \label{eq:fxy1}
\end{equation}
To see that $\nabla^2_{\bx\bx} f(0,\by)$ must be positive-semidefinite, we assume the contrary. We may then find a unit $d$-dimensional vector $\bu$ such that $\bu^\trop \nabla^2_{\bx\bx} f(0,\by) \bu = -G < 0$. Setting $\bx = \eps\bu$, we may then choose $\eps$ small enough that $\bu^\trop K(\eps\bu,\by) \bu < G$, so that again $f(\bx,\by)$ is negative and we have a contradiction. Finally, $\nabla^2_{\bx\by} f(0,\by) = \nabla^2_{\by\by} f(0,\by) = 0\; \forall \by$ follows directly from \eqref{eq:fxy1}, and we have established \ref{prop:gchi2}\ref{it:fdel2}.

We now prove \ref{prop:gchi2}\ref{it:fdel3} using a $2$nd-order Delta Method \citep{LehmannRomano:2005}. Let $\btheta \in \Theta_0$. Since $f(\btheta)$ and its gradient $\nabla\!f(\btheta)$ both vanish, the Taylor expansion of $f(\bvtheta_N)$ around $\btheta$ takes the form
\begin{equation}
	f(\bvtheta_N) = \tfrac12 \bracr{\bvtheta_N - \btheta}^\trop \Hess(\btheta) \bracr{\bvtheta_N - \btheta} + \bracr{\bvtheta_N - \btheta}^\trop K(\bvtheta_N) \bracr{\bvtheta_N - \btheta}
\end{equation}
where $\Hess(\btheta) = \nabla^2 \!f(\btheta)$ is the Hessian matrix of $f$ evaluated at $\btheta$, and $\lim_{\btheta' \to \btheta} K(\btheta') = 0$. Multiplying both sides by the sample size $N$, we have
\begin{equation}
	N f(\bvtheta_N) = \tfrac12 \bracs{\sqrt N\bracr{\bvtheta_N - \btheta}}^\trop \Hess(\btheta) \bracs{\sqrt N\bracr{\bvtheta_N - \btheta}} + \bracs{\sqrt N\bracr{\bvtheta_N - \btheta}}^\trop K(\bvtheta_N) \bracs{\sqrt N\bracr{\bvtheta_N - \btheta}}
\end{equation}
But by assumption $\sqrt N (\bvtheta_N-\btheta) \dconverge \bZ$ as $N \to \infty$, where $\bZ \sim \snormal(\bzero,\Omega)$. Thus, by the CMT \citep{van2000asymptotic}, we have
\begin{equation}
	N f(\bvtheta_N) \dconverge \tfrac12 \bZ^\trop \Hess(\btheta) \bZ \label{eq:thegenchi2}
\end{equation}
as $N \to \infty$, and \ref{prop:gchi2}\ref{it:fdel3} follows immediately from \eqref{eq:thegenchi2} and \ref{prop:gchi2}\ref{it:fdel2}.

\section{State-space solution for the reduced VAR model parameters} \label{sec:varpss}

Following \citet{Barnett:ssgc:2015}, given a VAR($p$) model \eqref{eq:varp} with parameters $\btheta = (A;\Sigma)$, we create an equivalent innovations-form state-space model \citep{HandD:2012}
\begin{subequations}
\begin{align}
	\bZ_{t+1} &= \sA \bZ_t + K \beps_t \\
	\bU_t &= A \bZ_t \hspace{2pt} + \hspace{9pt}\beps_t
\end{align}%
\end{subequations}
where
\begin{equation}
	\sA =
	\begin{bmatrix}
		A_1 & A_2 & \ldots & A_{p-1} & A_p \\
		I   & 0   & \ldots & 0       & 0   \\
		0   & I   & \ldots & 0       & 0   \\
		\vdots & \vdots & \ddots & \vdots & \vdots   \\
		0   & 0   & \ldots & I       & 0
	\end{bmatrix} \qquad\quad
	A =
	\begin{bmatrix}
		A_1 & A_2 & \ldots & A_{p-1} & A_p \\
	\end{bmatrix} \qquad\quad
	K =
	\begin{bmatrix}
		I \\ 0 \\ 0 \\ \vdots \\ 0
	\end{bmatrix}
\end{equation}
$\sA$ is the $pn \times pn$ state transition matrix [the companion matrix \eqref{eq:compA} for the VAR($p$) \eqref{eq:varp}], $K$ the $pn \times n$ Kalman gain matrix, and $A$ the $n \times pn$ observation matrix. As before, we use subscript $x$ for the indices $1,\ldots,n_x$, $y$ for the indices $n_x+1,\ldots,n$, and we use an asterisk to donate ``all indices''. The subprocess $\bX_t$ then satisfies the state-space model
\begin{subequations}
\begin{align}
	\bZ_{t+1} &= \sA \bZ_t \hspace{8pt}+ K \beps_t \\
	\bX_t &= A_{x*} \bZ_t + \hspace{10pt}\beps_{x,t}
\end{align}%
\end{subequations}
This state-space model is no longer in innovations form; we can, however \cite[see][]{Barnett:ssgc:2015} derive an innovations-form state-space model for $\bX_t$ by solving the discrete-time algebraic Riccati equation (DARE)
\begin{equation}
	P = \sA P \sA^\trop + \bSigma - \big(\sA P A_{x*}^\trop + \bSigma_{*x}\big) \big(A_{x*} P A_{x*}^\trop + \Sigma_{xx}\big)^{-1} \big(\sA P A_{x*}^\trop + \bSigma_{*x}\big)^\trop \label{eq:darebig}
\end{equation}
for $P$ (a $pn \times pn$ symmetric matrix), with
\begin{equation}
	\bSigma =
	\begin{bmatrix}
		\Sigma & 0 & \ldots & 0 \\
		0      & 0 & \ldots & 0 \\
		\vdots & \vdots & \ddots & \vdots \\
		0      & 0 & \ldots & 0
	\end{bmatrix} \qquad\quad
	\bSigma_{*x} =
	\begin{bmatrix}
		\Sigma_{*x} \\ 0 \\ \vdots \\ 0
	\end{bmatrix}
\end{equation}
which are, respectively, $pn \times pn$ and $pn \times n_x$. We note that under our assumptions, a stabilising solution for \eqref{eq:darebig} exists, and is unique \citep{Solo:2016}. Then
\begin{subequations}
\begin{align}
	\bZ_{t+1} &= \sA \bZ_t \hspace{8pt}+ K^\rrop \beps^\rrop_t \\
	\bX_t &= A_{x*} \bZ_t + \hspace{15pt}\beps^\rrop_t
\end{align} \label{eq:issred}%
\end{subequations}
is in innovations form, with innovations covariance matrix and Kalman gain matrix
\begin{subequations}
\begin{align}
	\Sigma^\rrop &= A_{x*} P A_{x*}^\trop + \Sigma_{xx} \\
	K^\rrop &= \big(\sA P A_{x*}^\trop + \bSigma_{*x}\big) \big[\Sigma^\rrop\big]^{-1} \label{eq:kgred}
\end{align}%
\end{subequations}
respectively. The $\beps^\rrop_t$ in \eqref{eq:issred} are precisely the reduced residuals in \eqref{eq:varr} and $\Sigma^\rrop = \expect{\beps^\rrop_t\beps^{\rrop\trop}_t}$ implicitly defines $V(\btheta)$ \eqref{eq:sigfun} as required for the single-regression GC statistic $F^\sre_{\bY\to\bX}(\btheta)$ \eqref{eq:popgcsr}.

We may in fact confirm that
\begin{equation}
	\Sigma^\rrop = A_{xy} \Pi A_{xy}^\trop + \Sigma_{xx} \label{eq:varredsig}
\end{equation}
where the $pn_y \times pn_y$ symmetric matrix $\Pi$ is the unique stabilising solution of the lower-dimensional DARE
\begin{equation}
	\Pi = \bA_{yy} \Pi \bA_{yy}^\trop + \bSigma_{yy} - \big(\bA_{yy} \Pi A_{xy}^\trop + \bSigma_{yx}\big) \big(A_{xy} \Pi A_{xy}^\trop + \Sigma_{xx}\big)^{-1} \big(\bA_{yy} \Pi A_{xy}^\trop + \bSigma_{yx}\big)^\trop \label{eq:daresmall}
\end{equation}
with
\begin{equation}
	\bA_{yy} =
	\begin{bmatrix}
		A_{1,yy} & A_{2,yy} & \ldots & A_{p-1,yy} & A_{p,yy} \\
		I   & 0   & \ldots & 0       & 0   \\
		0   & I   & \ldots & 0       & 0   \\
		\vdots & \vdots & \ddots & \vdots & \vdots   \\
		0   & 0   & \ldots & I       & 0
	\end{bmatrix} \qquad\quad
	A_{xy} =
	\begin{bmatrix}
		A_{1,xy} & A_{2,xy} & \ldots & A_{p-1,xy} & A_{p,xy}
	\end{bmatrix}
\end{equation}
which are, respectively, $pn_y \times pn_y$ and $n_x \times pn_y$, and
\begin{equation}
	\bSigma_{yy} =
	\begin{bmatrix}
		\Sigma_{yy} & 0 & \ldots & 0 \\
		0      & 0 & \ldots & 0 \\
		\vdots & \vdots & \ddots & \vdots \\
		0      & 0 & \ldots & 0
	\end{bmatrix} \qquad\quad
	\bSigma_{yx} =
	\begin{bmatrix}
		\Sigma_{yx} \\ 0 \\ \vdots \\ 0
	\end{bmatrix} \qquad\quad
\end{equation}
respectively, $pn_y \times pn_y$ and $pn_y \times n_x$. To see this, we may verify by substitution that if
\begin{equation}
	\Pi =
	\begin{bmatrix}
		\Pi_{11} & \cdots & \Pi_{1p} \\
		\vdots &        & \vdots \\
		\Pi_{p1} & \cdots & \Pi_{pp}
	\end{bmatrix}
\end{equation}
solves the reduced-dimension DARE \eqref{eq:daresmall}, where the $\Pi_{kl}$ are $n_y \times n_y$, then
\begin{equation}
	P =
	\begin{bmatrix}
		\begin{bmatrix}	0_{n_x \times n_x} & 0_{n_x \times n_y} \\ 0_{n_y \times n_x} & \Pi_{11} \end{bmatrix} & \cdots &
		\begin{bmatrix}	0_{n_x \times n_x} & 0_{n_x \times n_y} \\ 0_{n_y \times n_x} & \Pi_{1p} \end{bmatrix} \\
		\vdots &        & \vdots \\
		\begin{bmatrix}	0_{n_x \times n_x} & 0_{n_x \times n_y} \\ 0_{n_y \times n_x} & \Pi_{p1} \end{bmatrix} & \cdots &
		\begin{bmatrix}	0_{n_x \times n_x} & 0_{n_x \times n_y} \\ 0_{n_y \times n_x} & \Pi_{pp} \end{bmatrix}
	\end{bmatrix}
\end{equation}
solves the original DARE \eqref{eq:darebig}, and that $\Sigma^\rrop$ is given by \eqref{eq:varredsig}. We may also confirm that the Kalman gain matrix \eqref{eq:kgred} for the reduced model is given by
\begin{equation}
	K^\rrop =
	\begin{bmatrix}
		I_{n_x \times n_x} \\[2pt]
		\lK^\rrop_1 \\[2pt]
		0_{n_x \times n_x} \\[2pt]
		\lK^\rrop_2 \\
		\vdots \\
		0_{n_x \times n_x} \\[2pt]
		\lK^\rrop_p
	\end{bmatrix}
\end{equation}
where
\begin{equation}
	\lK^\rrop = \big(\sA_{yy} \Pi A_{xy}^\trop + \bSigma_{yx}\big) \big[\Sigma^\rrop\big]^{-1} =
	\begin{bmatrix}
		\lK^\rrop_1 \\[2pt]
		\lK^\rrop_2 \\
		\vdots \\[2pt]
		\lK^\rrop_p
	\end{bmatrix}
\end{equation}
(the $\lK_k^\rrop$ are  $n_y \times n_x$) is the Kalman gain matrix associated with the DARE \eqref{eq:daresmall}.

\section{Granger causality for a general bivariate \texorpdfstring{VAR($1$)}{VAR(1)} process} \label{sec:gcbivar1}

Consider the bivariate VAR($1$)
\begin{subequations}
\begin{align}
	X_t &= a_{xx} X_{t-1} + a_{xy} Y_{t-1} + \eps_{xt} \\
	Y_t &= a_{yx} X_{t-1} + a_{yy} Y_{t-1} + \eps_{yt}
\end{align}%
\end{subequations}
with parameters
\begin{equation}
	A = \twomat{a_{xx}}{a_{xy}}{a_{yx}}{a_{yy}}\,, \qquad \Sigma = \expect{\beps_t \beps_t^\trop} = \twomat{\sigma_{xx}}{\sigma_{xy}}{\sigma_{yx}}{\sigma_{yy}}
\end{equation}
The transfer function is then $\Psi(\omega) = \Phi(\omega)^{-1}$ with $\Phi(\omega) = I-Az$, and we have the spectral factorisation \eqref{eq:specfac} of the CPSD $S(\omega)$ of the VAR process, which holds for $z$ on the unit circle $|z| = 1$ in the complex plane.
Setting $\Delta(\omega) = |\Phi(\omega)|$ (determinant), we have
\begin{equation}
	\Psi(\omega) = \twomat{1-a_{xx}z}{-a_{xy}z}{-a_{yx}z}{1-a_{yy}z}^{-1} = \Delta(\omega)^{-1} \twomat{1-a_{yy}z}{a_{xy}z}{a_{yx}z}{1-a_{xx}z}
\end{equation}
This leads to
\begin{equation}
	S(\omega) = |\Delta(\omega)|^{-2} \twomat{1-a_{yy}z}{a_{xy}z}{a_{yx}z}{1-a_{xx}z} \twomat{\sigma_{xx}}{\sigma_{xy}}{\sigma_{yx}}{\sigma_{yy}} \twomat{1-a_{yy}\cz}{a_{yx}\cz}{a_{xy}\cz}{1-a_{xx}\cz}
\end{equation}
on $z = 1$, where $\cz$ is the complex conjugate of $z$; here,  for a complex variable $w$, $|w|$ denotes the norm $\sqrt{w\bar w}$.

We wish to calculate the GC $F_{Y \to X}$. If $v$ is the residuals variance for the VAR representation of the subprocess $X_t$, then the GC is just $F_{Y \to X} = \log v - \log{\sigma_{xx}}$. To solve for $v$ we could use the state-space method of \citet{Barnett:ssgc:2015}, but here we use an explicit spectral factorisation for the CPSD $S_{xx}(\omega)$ of $X_t$.

Let $\psi(\omega)$ be the transfer function of the process $X_t$. It then satisfies the spectral factorisation
\begin{equation}
	S_{xx}(\omega) = v |\psi(\omega)|^2
\end{equation}
with $\psi(0) = 1$. We may now calculate (we denote terms we don't need by ``$\cdots$'').
\begin{align*}
	S(\omega)
	&= |\Delta(\omega)|^{-2} \twomat{1-a_{yy}z}{a_{xy}z}\cdots\cdots \twomat{\sigma_{xx}}{\sigma_{xy}}{\sigma_{yx}}{\sigma_{yy}} \twomat{1-a_{yy}\cz}\cdots{a_{xy}\cz}\cdots \\
	&= |\Delta(\omega)|^{-2} \twomat{1-a_{yy}z}{a_{xy}z}\cdots\cdots \twomat{\sigma_{xx}(1-a_{yy}\cz)+\sigma_{xy} a_{xy}\cz}\cdots{\sigma_{yx}(1-a_{yy}\cz)+\sigma_{yy} a_{xy}\cz}\cdots
\end{align*}
We now calculate [with $z =e^{-i\omega}$ so that $\real{z} = \frac12(z+\cz) = \cos\omega$]:
\begin{align*}
	S_{xx}(\omega)
	&= |\Delta(\omega)|^{-2} \bracc{(1-a_{yy}z) [\sigma_{xx}(1-a_{yy}\cz)+\sigma_{xy} a_{xy}\cz)] + a_{xy}z [\sigma_{yx}(1-a_{yy}\cz)+\sigma_{yy} a_{xy}\cz)]} \\
	&= |\Delta(\omega)|^{-2} \bracc{\sigma_{xx} |1-a_{yy}z|^2 + \sigma_{xy} a_{xy} [(1-a_{yy}z)\cz+ (1-a_{yy}\cz)z] + \sigma_{yy} a_{xy}^2 z\cz  } \\
	&= |\Delta(\omega)|^{-2} \bracc{\sigma_{xx} [1-a_{yy} (z+\cz)+a_{yy}^2] + \sigma_{xy} a_{xy} (z+\cz-2a_{yy}) + \sigma_{yy} a_{xy}^2} \\
	&= |\Delta(\omega)|^{-2} \bracc{\sigma_{xx} [1-2a_{yy}\cos\omega+a_{yy}^2] + 2\sigma_{xy} a_{xy} (\cos\omega-a_{yy}) + \sigma_{yy} a_{xy}^2}
\end{align*}
and finally
\begin{equation}
	S_{xx}(\omega) = |\Delta(\omega)|^{-2} (P - Q\cos\omega) \label{eq:bvSxx}
\end{equation}
where we have set
\begin{subequations}
\begin{align}
	P &= \sigma_{xx} (1+a_{yy}^2) - 2\sigma_{xy} a_{xy} a_{yy} + \sigma_{yy} a_{xy}^2 \\
	Q &= 2 (\sigma_{xx} a_{yy} - \sigma_{xy} a_{xy})
\end{align}%
\end{subequations}
The form of this expression suggests that the transfer function $\psi(\omega)$ should take the form
\begin{equation}
	\psi(\omega) = \Delta(\omega)^{-1} (1-b z)
\end{equation}
for some constant $b$. Note that this implies that $X_t$ is ARMA(2,1). Then $|\psi(\omega)|^2 = |\Delta(\omega)|^{-2} (1+b^2 - 2b\cos\omega)$ and the spectral factorisation $S_{xx}(\omega) = v |\psi(\omega)|^2$ now reads:
\begin{equation}
	v (1+b^2 - 2b\cos\omega) = P - Q\cos\omega
\end{equation}
This must hold for at point on the unit circle---\ie, for all $\omega$---so we must have
\begin{align}
	v (1+b^2) &= P \\
	v b &= \tfrac12 Q
\end{align}
We may now solve for $v$. The second equation gives $v^2 b^2 = \tfrac14 Q^2$, so multiplying the first equation through by $v$ we obtain the quadratic equation for $v$:
\begin{equation}
	v^2 -Pv + \tfrac14 Q^2 = 0
\end{equation}
with solutions
\begin{equation}
	v = \frac12 \bracr{P \pm \sqrt{P^2-Q^2}}
\end{equation}
We need to take the ``$+$'' solution, as this yields the correct (zero) result for the null case $a_{xy} = 0$. Note that only the $Y\to X$ ``causal'' coefficient $a_{xy}$ and the $Y$ autoregressive coefficient $a_{yy}$ appear in the expression for $F_{Y \to X}$.

From \eqref{eq:sgc}, the spectral GC from $Y \to X$ is
\begin{equation}
	f_{Y \to X}(\omega) = \log\frac{P-Q\cos\omega}{P-Q\cos\omega - a_{xy}^2 \sigma_{yy|x}}
\end{equation}
where
\begin{equation}
	\sigma_{yy|x} = \sigma_{yy} - \frac{\sigma_{xy}^2}{\sigma_{xx}} = \sigma_{yy}\bracr{1-\kappa^2}
\end{equation}
with $\displaystyle \kappa = \frac{\sigma_{xy}}{\sqrt{\sigma_{xx}\sigma_{yy}}}$ the residuals correlation.

For the sampling distributions, we shall also need the (inverse of) the covariance matrix $\Gamma_0$ of the process $[X_t^\trop \; Y_t^\trop]^\trop$ on the null space $a_{xy} = 0$. Solving the DLYAP equation $\Gamma_0 - A\Gamma_0 A^\trop = \Sigma$ for $\Gamma_0 = \twomat prrq$ yields
\begin{subequations}
\begin{align}
	p &= \bracr{1-a_{xx}^2}^{-1} \sigma_{xx} \\
	r &= (1-a_{xx}a_{yy})^{-1} \bracs{\sigma_{xy} + a_{xx}a_{yx} \bracr{1-a_{xx}^2}^{-1} \sigma_{xx}} \\
	q &= \bracr{1-a_{yy}^2}^{-1} \bracs{\sigma_{yy} + 2a_{yy}a_{yx} (1-a_{xx}a_{yy})^{-1} \sigma_{xy} + a_{yx}^2 (1+a_{xx}a_{yy}) \bracr{1-a_{xx}^2}^{-1} (1-a_{xx}a_{yy})^{-1} \sigma_{xx}}
\end{align}%
\end{subequations}
and we have in particular
\begin{equation}
	\omega_{yy} = [\Gamma_0^{-1}]_{yy} = \frac p{pq-r^2} \label{eq:omyy}
\end{equation}
Note also that in the null case $a_{xy} = 0$, the spectral radius is $\rho = \max\bracr{|a_{xx}|, |a_{yy}|}$.

\section{Parametrised sampling of the VAR model space} \label{sec:ranvar}

Consider, for given number of variables $n$ and model order $p$, the parameter space $\Theta = \{(A,\Sigma) : A \text{ is } n \times pn \text{ with } \rho(A) < 1\,, \Sigma \text{ is } n \times n \text{ positive-definite}\}$ of VAR($p$) models. Firstly, we note that the residuals covariance matrix $\Sigma$ can be taken to be a \emph{correlation} matrix; this can always be achieved by a rescaling of variables leaving Granger causalities invariant. Further GC invariances under linear transformation of variables \citep{Barrett:2010} allow further effective dimensional reduction of $\Theta$; however, even under these general transformations, and under the constraint $\rho(A) < 1$, the quotient space of $\Theta$ has infinite Lebesgue measure\footnote{Although the space of $n \times n$ correlation matrices has finite measure.}; thus we cannot generate uniform variates (it is questionable whether this would in any case be appropriate to a given empirical scenario). Here we utilise a practical and flexible scheme for generation of variates on $\Theta$, parametrised by spectral radius $\rho$, log-generalised correlation\footnote{For Gaussian covariance matrices, log-generalised correlation coincides with \emph{multi-information} \citep{StudenyVejnarova:1998}. If $R = (\rho_{ij})$ is a correlation matrix with all $\rho_{ij} \ll 1$ for $i \ne j$, then $-\log|R| \approx \sum_{i<j} \rho_{ij}^2$.} $\gamma = -\log|\Sigma| + \sum_i \log\Sigma_{ii}$, and population GC $F = F_{\bY\to\bX}(\btheta)$, all of which have a critical impact on GC sampling distributions.

To generate a random correlation matrix $\Sigma$ of dimension $n$ with given generalised correlation $\gamma$, we use the following algorithm:
\begin{enumerate}
	\item Starting with an $n \times n$ matrix with components iid $\sim\snormal(0,1)$, we compute its QR-decomposition $[Q,R]$. The matrix $M_{ij} = Q_{ij}\cdot\sign(R_{jj})$ is then a random orthogonal matrix.
	\item Next we create a random $n$-dimensional variance vector $\bv$ with components $v_i$ iid $\sim\chi^2(1)$. The matrix $ V = M\cdot\diag\bv\cdot M^\trop$ is then positive-definite, and for the corresponding correlation matrix $\Sigma_{ij} = V_{ij}/\sqrt{V_{ii}V_{jj}}$ we have $\gamma^* = -\sum_i \log v_i + \sum_i \log V_{ii}$. \\[-0.7em]

	If necessary, we repeat steps 1,2 until $\gamma^* \ge \gamma$ (this may fail if $\gamma$ is too large).
	\item Using a binary chop, we find a constant $c$ such that, iteratively replacing $\bv \leftarrow \bv+c$,  $\gamma^*$ falls within an acceptable tolerance of $\gamma$ (this generally converges). The correlation matrix $\Sigma$ is then returned,
\end{enumerate}
For a VAR coefficients matrix sequence $A = [A_1\;A_2\;\ldots\;A_p]$, the spectral radius $\rho(A)$ is given by \eqref{eq:specrad}. If $\lambda$ is a constant, it is easy to show that if $\tilde A$ is the sequence $[\lambda A_1\;\lambda^2 A_2\;\ldots\;\lambda^p A_p]$, then $\rho\big(\tilde A\big) = \lambda\rho(A)$. Thus any VAR coefficients sequence may be exponentially weighted so that its spectral radius takes a given value. Such weighting, however, has the side-effect of exponential decay of the $A_k$ with lag $k$, which is, anecdotally, unrealistic\footnote{At least, in the authors' experience, for neural or econometric data.}. We observe empirically that we can compensate for this decay reasonably consistently across number of variables and model orders by scaling all coefficients by $A_k$ by $e^{-\sqrt p w}$ for some constant $w$; here we choose $w = 1$, which generally achieves a more realistic gradual and approximately linear decay. To generate a random VAR model with given generalised correlation $\gamma$ and given spectral radius $\rho$, our procedure is as follows:
\begin{enumerate}
	\item We generate a random correlation matrix $\Sigma$ with generalised correlation $\gamma$ as described above.
	\item We generate $p$ $n \times n$ coefficient matrices $A_k$ with components iid $\sim\snormal(0,1)$. The $A_k$ are the weighted uniformly by $e^{-\sqrt p w}$.
\end{enumerate}
To enforce the null condition $A_{k,xy} = 0$,
\begin{enumerate}
	\setcounter{enumi}{2}
	\item The $A_{k,xy}$ components are all set to zero.
	\item The $A_k$ coefficients sequence is scaled exponentially by an appropriate constant $\lambda$, so as to achieve the given spectral radius $\rho$.
\end{enumerate}
To instead enforce a given (non-null) population GC value $F$,
\begin{enumerate}
	\setcounter{enumi}{2}
	\item The $A_{k,xy}$ components are scaled uniformly by a constant $c$.
	\item The $A_k$ coefficients sequence is scaled exponentially by an appropriate constant $\lambda$, so as to achieve the given spectral radius $\rho$. \\[-0.7em]

	Under steps $3,4$ the population GC depends monotonically on $c$; consequently,
	\item We perform a binary chop on $c$, iterating steps $3,4$ until the GC is within an acceptable tolerance of $F$ (this generally converges quickly).
\end{enumerate}
In all simulations except for the bivariate model (\secref{sec:examp}), we used $\gamma = 1$; spectral radii and population GC values are as indicated in the plots. Convergence tolerances were set to\footnote{Approximately $1.5 \times 10^{-8}$ under the IEEE 754-2008 binary$64$ floating point standard.} $\sqrt{\text{machine }\eps}$.

\vspace{2em}
\bibliographystyle{apalike}
\bibliography{gcsreg}

\begin{thebibliography}{}

\bibitem[Amari, 2016]{Amari:2016}
Amari, S. (2016).
\newblock {\em Information Geometry and its Applications}.
\newblock Springer, Japan.

\bibitem[Barnett et~al., 2009]{Barnett:tegc:2009}
Barnett, L., Barrett, A.~B., and Seth, A.~K. (2009).
\newblock Granger causality and transfer entropy are equivalent for {G}aussian
  variables.
\newblock {\em Phys. Rev. Lett.}, 103(23):0238701.

\bibitem[Barnett et~al., 2018a]{BarnettEtal2018a}
Barnett, L., Barrett, A.~B., and Seth, A.~K. (2018a).
\newblock Misunderstandings regarding the application of {G}ranger causality in
  neuroscience.
\newblock {\em Proc. Natl. Acad. Sci. USA}, 115(29):E6676--E6677.

\bibitem[Barnett et~al., 2018b]{BarnettEtal2018b}
Barnett, L., Barrett, A.~B., and Seth, A.~K. (2018b).
\newblock Solved problems for {G}ranger causality in neuroscience: {A} response
  to {S}tokes and {P}urdon.
\newblock {\em NeuroImage}, 178:744--748.

\bibitem[Barnett and Bossomaier, 2013]{Barnett:teml:2012}
Barnett, L. and Bossomaier, T. (2013).
\newblock Transfer entropy as a log-likelihood ratio.
\newblock {\em Phys. Rev. Lett.}, 109(13):0138105.

\bibitem[Barnett and Seth, 2011]{Barnett:gcfilt:2011}
Barnett, L. and Seth, A.~K. (2011).
\newblock Behaviour of {G}ranger causality under filtering: {T}heoretical
  invariance and practical application.
\newblock {\em J. Neurosci. Methods}, 201(2):404--419.

\bibitem[Barnett and Seth, 2014]{Barnett:mvgc:2014}
Barnett, L. and Seth, A.~K. (2014).
\newblock The {MVGC} multivariate {G}ranger causality toolbox: {A} new approach
  to {G}ranger-causal inference.
\newblock {\em J. Neurosci. Methods}, 223:50--68.

\bibitem[Barnett and Seth, 2015]{Barnett:ssgc:2015}
Barnett, L. and Seth, A.~K. (2015).
\newblock Granger causality for state-space models.
\newblock {\em Phys. Rev. E (Rapid Communications)}, 91(4):040101(R).

\bibitem[Barrett et~al., 2010]{Barrett:2010}
Barrett, A.~B., Barnett, L., and Seth, A.~K. (2010).
\newblock Multivariate {G}ranger causality and generalized variance.
\newblock {\em Phys. Rev. E}, 81(4):041907.

\bibitem[Chen et~al., 2006]{Chen:2006}
Chen, Y., Bressler, S.~L., and Ding, M. (2006).
\newblock Frequency decomposition of conditional {G}ranger causality and
  application to multivariate neural field potential data.
\newblock {\em J. Neuro. Methods}, 150:228--237.

\bibitem[Dhamala et~al., 2018]{DhamalaEtal:2018}
Dhamala, M., Liang, H., Bressler, S.~L., and Ding, M. (2018).
\newblock Granger-{G}eweke causality: {E}stimation and interpretation.
\newblock {\em NeuroImage}, 175:460--463.

\bibitem[Dhamala et~al., 2008a]{Dhamala:2008b}
Dhamala, M., Rangarajan, G., and Ding, M. (2008a).
\newblock Analyzing information flow in brain networks with nonparametric
  {G}ranger causality.
\newblock {\em Neuroimage}, 41:354--362.

\bibitem[Dhamala et~al., 2008b]{Dhamala:2008a}
Dhamala, M., Rangarajan, G., and Ding, M. (2008b).
\newblock Estimating {G}ranger causality from {F}ourier and wavelet transforms
  of time series data.
\newblock {\em Phys. Rev. Lett.}, 100:018701.

\bibitem[Ding et~al., 2006]{DingEtal:2006}
Ding, M., Chen, Y., and Bressler, S.~L. (2006).
\newblock Granger causality: Basic theory and application to neuroscience.
\newblock In Schelter, B., Winterhalder, M., and Timmer, J., editors, {\em
  Handbook of Time Series Analysis}, pages 437--460. Wiley-VCH Verlag GmbH \&
  Co. KGaA.

\bibitem[Doob, 1953]{Doob:1953}
Doob, J. (1953).
\newblock {\em Stochastic Processes}.
\newblock John Wiley, New York.

\bibitem[Faes et~al., 2017]{FaesEtal:2017}
Faes, L., Stramaglia, S., and Marinazzo, D. (2017).
\newblock On the interpretability and computational reliability of
  frequency-domain {G}ranger causality.
\newblock {\em F1000Research}, 6(1710).
\newblock Version 1; Referees: 2 approved.

\bibitem[Geweke, 1982]{Geweke:1982}
Geweke, J. (1982).
\newblock Measurement of linear dependence and feedback between multiple time
  series.
\newblock {\em J. Am. Stat. Assoc.}, 77(378):304--313.

\bibitem[Geweke, 1984]{Geweke:1984}
Geweke, J. (1984).
\newblock Measures of conditional linear dependence and feedback between time
  series.
\newblock {\em J. Am. Stat. Assoc.}, 79(388):907--915.

\bibitem[Hamilton, 1994]{Hamilton:1994}
Hamilton, J.~D. (1994).
\newblock {\em Time Series Analysis}.
\newblock Princeton University Press, Princeton, NJ.

\bibitem[Hannan and Deistler, 2012]{HandD:2012}
Hannan, E.~J. and Deistler, M. (2012).
\newblock {\em The Statistical Theory of Linear Systems}.
\newblock SIAM, Philadelphia, PA, USA.

\bibitem[Hannan and Quinn, 1979]{HandQ:1979}
Hannan, E.~J. and Quinn, B.~G. (1979).
\newblock The determination of the order of an autoregression.
\newblock {\em J. Roy. Stat. Soc. B Met.}, 41(2):190--195.

\bibitem[Jones, 1983]{Jones:1983}
Jones, D.~A. (1983).
\newblock Statistical analysis of empirical models fitted by optimization.
\newblock {\em Biometrika}, 70(1):67--88.

\bibitem[Klein et~al., 2000]{KleinEtal:2000}
Klein, A., M{\'e}lard, G., and Zahaf, T. (2000).
\newblock Construction of the exact {F}isher information matrix of {G}aussian
  time series models by means of matrix differential rules.
\newblock {\em Linear Algebra Appl.}, 321(1):209--232.
\newblock Eighth Special Issue on Linear Algebra and Statistics.

\bibitem[Lehmann and Romano, 2005]{LehmannRomano:2005}
Lehmann, E.~L. and Romano, J.~P. (2005).
\newblock {\em Testing Statistical Hypotheses}.
\newblock Springer Science+Business Media, LLC, New York, NY, USA, 3rd edition.

\bibitem[L\"utkepohl, 1993]{Lutkepohl:1993}
L\"utkepohl, H. (1993).
\newblock Testing for causation between two variables in higher dimensional
  {VAR} models.
\newblock In Schneewei{\ss}, H. and Zimmerman, K., editors, {\em Studies in
  Applied Econometrics}, pages 75--91. Physica-Verlag HD, Heidelberg.

\bibitem[L\"utkepohl, 2005]{Lutkepohl:2005}
L\"utkepohl, H. (2005).
\newblock {\em New Introduction to Multiple Time Series Analysis}.
\newblock Springer-Verlag, Berlin.

\bibitem[Masani, 1966]{Masani:1966}
Masani, P. (1966).
\newblock Recent trends in multivariate prediction theory.
\newblock In Krishnaiah, P.~R., editor, {\em Multivariate Analysis}, pages
  351--382. Academic Press, New York.

\bibitem[Mc{Q}uarrie and Tsai, 1998]{McQuarrie:1998}
Mc{Q}uarrie, A. D.~R. and Tsai, C.-L. (1998).
\newblock {\em Regression and Time Series Model Selection}.
\newblock World Scientific Publishing, Singapore.

\bibitem[Mitra and Bokil, 2008]{MitraBokil:2008}
Mitra, P.~P. and Bokil, H. (2008).
\newblock {\em Observed Brain Dynamics}.
\newblock Oxford University Press, New York.

\bibitem[Mohsenipour, 2012]{Mohsenipour:2012}
Mohsenipour, A.~A. (2012).
\newblock {\em On the Distribution of Quadratic Expressions in Various Types of
  Random Vectors}.
\newblock PhD thesis, The University of Western Ontario, Electronic Thesis and
  Dissertation Repository, 955. \url{https://ir.lib.uwo.ca/etd/955}.

\bibitem[Neyman and Pearson, 1933]{NeymanPearson:1933}
Neyman, J. and Pearson, E.~S. (1933).
\newblock On the problem of the most efficient tests of statistical hypotheses.
\newblock {\em Phil. Trans. R. Soc. A}, 231:289--337.

\bibitem[Rozanov, 1967]{Rozanov:1967}
Rozanov, Y.~A. (1967).
\newblock {\em Stationary Random Processes}.
\newblock Holden-Day, San Francisco.

\bibitem[Solo, 2016]{Solo:2016}
Solo, V. (2016).
\newblock State-space analysis of {G}ranger-{G}eweke causality measures with
  application to {fMRI}.
\newblock {\em Neural Comput.}, 28(5):914--949.

\bibitem[Stokes and Purdon, 2017]{StokesPurdon:2017}
Stokes, P.~A. and Purdon, P.~L. (2017).
\newblock A study of problems encountered in {G}ranger causality analysis from
  a neuroscience perspective.
\newblock {\em Proc. Natl. Acad. Sci. USA}, 114(34):7063--7072.

\bibitem[Stokes and Purdon, 2018]{StokesPurdon:2018}
Stokes, P.~A. and Purdon, P.~L. (2018).
\newblock Correction for {S}tokes and {P}urdon, {A} study of problems
  encountered in {G}ranger causality analysis from a neuroscience perspective.
\newblock {\em Proc. Natl. Acad. Sci. USA}, 115(29):E6964--E6964.

\bibitem[Studen{\'y} and Vejnarov{\'a}, 1998]{StudenyVejnarova:1998}
Studen{\'y}, M. and Vejnarov{\'a}, J. (1998).
\newblock The multiinformation function as a tool for measuring stochastic
  dependence.
\newblock In Jordan, M.~I., editor, {\em Learning in Graphical Models}, pages
  261--297. Springer Netherlands, Dordrecht.

\bibitem[Van~der Vaart, 2000]{van2000asymptotic}
Van~der Vaart, A.~W. (2000).
\newblock {\em Asymptotic statistics}, volume~3.
\newblock Cambridge University Press.

\bibitem[Wald, 1943]{Wald:1943}
Wald, A. (1943).
\newblock Tests of statistical hypotheses concerning several parameters when
  the number of observations is large.
\newblock {\em T. Am. Math. Soc.}, 54(3):426--482.

\bibitem[Whittle, 1963]{Whittle:1963}
Whittle, P. (1963).
\newblock On the fitting of multivariate autoregressions, and the approximate
  canonical factorization of a spectral density matrix.
\newblock {\em Biometrika}, 50(1,2):129--134.

\bibitem[Wiener and Masani, 1957]{WienerMasani:1957}
Wiener, N. and Masani, P. (1957).
\newblock The prediction theory of multivariate stochastic processes: {I}.
  {T}he regularity condition.
\newblock {\em Acta Math.}, 98:111--150.

\bibitem[Wilks, 1932]{Wilks:1932}
Wilks, S.~S. (1932).
\newblock Certain generalizations in the analysis of variance.
\newblock {\em Biometrika}, 24:471--494.

\bibitem[Wilks, 1938]{Wilks:1938}
Wilks, S.~S. (1938).
\newblock The large-sample distribution of the likelihood ratio for testing
  composite hypotheses.
\newblock {\em Ann. Math. Stat.}, 6(1):60--62.

\bibitem[Wilson, 1972]{Wilson:1972}
Wilson, G.~T. (1972).
\newblock The factorization of matricial spectral densities.
\newblock {\em SIAM J. Appl. Math.}, 23(4):420--426.

\end{thebibliography}


\end{document}